\documentclass[3p]{elsarticle}

\usepackage{setspace}

\usepackage{amsmath,amsfonts,amssymb,amsbsy,amscd, amsthm}
\usepackage{graphicx}
\usepackage{mathrsfs,array}
\usepackage{verbatim} 
\usepackage{tikz}
\usepackage{subcaption}
\usepackage{pstricks-add}
\usepackage[section,subsection,subsubsection]{extraplaceins}

\newcommand{\vertiii}[1]{{\left\vert\kern-0.25ex\left\vert\kern-0.25ex\left\vert #1 
    \right\vert\kern-0.25ex\right\vert\kern-0.25ex\right\vert}}

\newtheorem{lem}{Lemma}
\newtheorem{theorem}[lem]{Theorem}
\newtheorem{definition}[lem]{Definition}
\newtheorem{prop}[lem]{Proposition}
\newtheorem{corollary}[lem]{Corollary}

\numberwithin{lem}{section}

\newdefinition{rmk}[lem]{Remark}

\newcommand{\eps}{\varepsilon}

\newcommand{\nn}{\nonumber}

\newcommand{\pvarphiF}{\dfrac{\partial F}{\partial \varphi}}
\newcommand{\pthetaF}{\dfrac{\partial F}{\partial \theta}}
\newcommand{\pvarphivarphiF}{\dfrac{\partial^2F}{\partial \varphi^2}}
\newcommand{\pthetathetaF}{\dfrac{\partial^2F}{\partial \theta^2}}
\newcommand{\pvarphithetaF}{\dfrac{\partial^2F}{\partial \varphi \partial\theta}}

\def\be{\begin{equation}}
\def\ee{\end{equation}}
\def\ba{\begin{array}}
\def\ea{\end{array}}
\def\bea{\begin{eqnarray}}
\def\eea{\end{eqnarray}}
\def\beas{\begin{eqnarray*}}
\def\eeas{\end{eqnarray*}}
\newcommand{\bseq}{\begin{subequations}}
\newcommand{\eseq}{\end{subequations}}

\biboptions{numbers,sort&compress}

\usepackage{hyperref}
\hypersetup{
    colorlinks,
    citecolor=ck,
    filecolor=ck,
    linkcolor=ck,
    urlcolor=ck
}

\usepackage{color}

\newcommand{\Ido}[1] { {\color{violet}  #1} }
\newcommand{\Derive}[1] { {\color{green}  #1} }
\newcommand{\TBD}[1] { {\color{red} #1} }

\renewcommand{\Ido}[1] {}
\renewcommand{\Derive}[1] {}
\renewcommand{\TBD}[1] {}

\usepackage[final]{changes}
\definechangesauthor[color=blue]{Rev.1}
\definechangesauthor[color=red]{Rev.2}
\definechangesauthor[color=purple]{Authors}

\begin{document}

\begin{frontmatter}

\title{Variation entropy: a continuous local generalization of the TVD property using entropy principles.}
\author{M.F.P. ten Eikelder\corref{cor1}}
\cortext[cor1]{Corresponding author}
\ead{m.f.p.teneikelder@tudelft.nl}
\author{I. Akkerman}
\ead{i.akkerman@tudelft.nl}

\address{Delft University of Technology, Department of Mechanical, Maritime and Materials Engineering, P.O. Box 5, 2600 AA Delft, The Netherlands}
\date{\today}

\journal{Computer methods in Applied Mechanics and Engineering}

\begin{abstract}
This paper presents the notion of a variation entropy. 
This concept is an entropy framework for the gradient of the solution of a conservation law instead of on the solution itself.
It appears that all semi-norms are admissible variation entropies.
This provides insight into the total variation diminishing property and justifies it from entropy principles.
The framework allows to derive new local variation diminishing properties in the continuous form.
This can facilitate the design of new numerical methods for problems that contain discontinuities.
\end{abstract}

\end{frontmatter}

\section{Introduction}
\label{sec:intro}

Violent disturbances emerging from sudden changes in velocity, pressure and temperature, known as shock waves appear everywhere in nature, science and industrial applications. 
Examples are water--air flows, supersonic flights, the water hammer phenomena, shock--bubble interaction, material impact and sudden changes in crowd dynamics. 
The behavior of these phenomena is usually governed by nonlinear conservation laws.
The development of numerical techniques for the solution procedure of conservation laws is challenging because higher-order methods produce oscillations near shocks. 
There exists a large class of numerical methods which aim to tackle these oscillations via reducing to first-order spatial accuracy at the shock wave. 
These techniques augment the numerical method in one way or another with artificial diffusion or viscosity in the shock wave region. 

Most numerical methods developed for problems involving shock waves use finite-difference or finite-volume approaches.
These methods are often well-established and show good performance in numerical computations.
The employed mechanisms can often be linked to one of the following.
The concept of flux limiters (MUSCL), see e.g. \citep{van1979towards, sweby1984high, van1997towards, toro2000centred, koren1993robust}, reduces the scheme at the shock to first-order by adding diffusion.
The monotonicity property introduced by Harten in 1983 \citep{harten1983high} precludes the creation of local extrema and ensures that local minima (maxima) are non-decreasing (non-increasing).
Perhaps the most relevant in some numerical simulations is the maximum principle, see \citep{harten1983high} or the more recent work \citep{zhang2010genuinely, mehmetoglu2012maximum}.
This principle states that the solution values remain between the minimum and maximum of the initial condition.
This is in particular important in simulations of physical quantities that should remain non-negative, e.g. densities and also, in the case of two-fluid problems, volume-fractions.
A negative density or a volume-fraction exceeding the zero-one range can directly lead to a blow-up of the simulations.
Therefore numerical methods that preclude this by design are often sought after, see e.g. for compressible two-fluid flow simulations \cite{shen2017maximum, ten2017acoustic, daude2016computation}.

Possibly the most famous property is the total variation diminishing (TVD) property \cite{harten1983high, harten1984class, shu1988total}. 
The total variation diminishing schemes preclude the growth of the total variation of the solution.
These methods ensure that the numerical solution $\phi$ of a PDE satisfies
\begin{align}
  {\rm TV}(\phi^{n+1}) \leq {\rm TV}(\phi^{n}),
\end{align}
where $n$ denotes the time-level.
Here the total variation ${\rm TV}$ is defined in one and two dimensions respectively as:
\begin{subequations}
  \begin{alignat}{2}
  {\rm TV}(\phi) =& \sum_j | \phi_{j+1}-\phi_j|,\\
  {\rm TV}(\phi) =& \sum_{j,k} \left(\Delta y | \phi_{j+1,k}-\phi_{j,k}|+ \Delta x |\phi_{j,k+1}-\phi_{j,k}|\right), \label{eq: TVD FD, 2D}
  \end{alignat}
\end{subequations}
with the subscript the spatial index and with $|\cdot|$ the absolute value.
This definition is based on the discrete approximation of the $L_1$-norm of the gradient.
It is important to emphasize that it represents a global property that assumes an underlying Cartesian grid that is possibly non-uniform and does not have a straightforward equivalent for unstructured grids.
Furthermore, the TVD property does not readily extend to multiple dimensions.
This has been the motivation for the development of local extremum diminishing (LED) schemes \cite{jameson1989computational,jameson1995positive}. 
In the two-dimensional case, the definition of the total variation depends on the orientation of the grid, i.e. it would change if one would rotate the grid.
This is undesirable since (i) this total variation does not have a continuous counterpart and (ii) generally one does not want the numerical method to depend on the coordinate system (objectivity).
Note that all of the above mentioned features and properties are satisfied by the exact solution and are in the discrete setting closely linked.
In particular cases (e.g. one-dimensional, scalar) one implies another \citep{harten1983high}.

In the framework of finite element methods several stabilized methods have been proposed to deal with spurious wiggles in the solution profiles of convection-dominated problems. 
The well-known methods are the Streamline upwind-Petrov Galerkin (SUPG) method \citep{BroHug82}, the Galerkin/least-squares method \citep{hughes1989new8} and the variational multiscale method \citep{Hug95, Hug96, Hug98}. 
The latter method offers a rich prospect for design new stabilized methods and has gained a lot of attention recently \citep{colomes2015assessment, yan2017new, bayona2018variational,EiAk17i, EiAk17ii}.
In the direction of TVD schemes and maximum principles, several VMS methods have been proposed.
For example a total variation bounding constraint \citep{evans2010VMSDC} and the maximum principle \citep{evans2009enforcement} has been enforced  in the VMS context.
When shocks waves form, plain stabilized methods are not sufficient and additional dissipation mechanisms are necessary.
These mechanisms are called discontinuity capturing (DC) operators \citep{hughes1986new4, leBeau1991finite}, and are sometimes residual-based \citep{tezduyar2004finitea, BaCaTeHu07, ABKF11} or entropy-residual based \citep{guermond2011entropy}.
We refer for an overview of DC to the review papers \citep{hughes2010stabilized, JoKn07, john2008spurious}.

To single out the physically relevant solutions the concept of entropy solutions has been proposed by Kruvzkov in his seminal 1970 paper \citep{kruvzkov1970first}. 
The entropy solution is a limiting case of a generalized solution which perturbs the conservation law with a diffusion term and can be used to prove existence, uniqueness and stability theorems.
In the case of systems of conservation laws, Friedrichs, Kurt and Lax show in 1971 that if an additional conserved quantity is a convex function of the solution then the system of equations can be symmetrized and provides a corresponding entropy inequality \citep{friedrichs1971systems}. 
Harten continues the research on symmetrizability of systems of conservation laws which possess entropy functions \citep{harten1983symmetric}. 
Additionally, he provides symmetric formulations in conservative variables for the Euler equations of gas dynamics. 
Tadmor shows a year later, in 1984, that the concepts of symmetrizability, having an entropy function and having a so-called skew-selfadjoint form are equivalent \citep{tadmor1984skew}. 
Furthermore Tadmor identifies in \citep{tadmor1987entropy} that any symmetric system of conservation laws is equipped with a one-parameter family of entropy functions. 
The work of Harten and Tadmor has been generalized by Hughes et al. \citep{hughes1986new1} to the compressible Navier-Stokes equations with heat conduction effects. 
The corresponding finite element schemes satisfy by design the second law of thermodynamics, see also \citep{chalot1990symmetrization}.

Although total variation diminishing schemes have proven their power and relevance, their use seems to be restricted to finite-difference/finite-volume discretizations and a generic finite element variant seems to be missing.
Moreover, the different concepts of total variation diminishing schemes and entropy solutions/entropy variables both target to improve the solution quality at shock waves.
Despite that they serve the same goal, a clear connection (on the continuous level) is missing.
These observations led to ponder the following two questions: 
\begin{itemize}
 \item \textit{How can we construct a local continuous generalization of the TVD stability condition?}
 \item \textit{Is there a connection between entropy solutions and the TVD property?}
\end{itemize}

The current paper aims to answer these questions.
To that purpose, we introduce a new stability concept which we call \textit{variation entropy}.
Similar to the well-known entropy concept, \textit{variation entropy solutions} are those solutions for which an additional quantity is conserved or dissipated.
The fundamental difference is that a variation entropy is a function of the \textit{gradient of the solution} rather than the solution itself.
Variation entropy solutions are presented in the continuous setting and are as such not restricted to a particular discretization.
Therefore, in contrast to the TVD property, numerical methods employing variation entropy concept may be based on a variational setting (e.g. finite element methods).
An important observation is that the TVD property may be derived from a specific variation entropy solution.

We summarize the main definitions and results formally (i.e. up to regularization).
Consider a conservation law for the unknown $\phi$ and flux $\mathbf{f}=\mathbf{f}(\phi)$:
\begin{align*}
    \partial_t \phi + \nabla \cdot \mathbf{f} = 0.
\end{align*}
The \textit{variation entropy} and the \textit{variation entropy solution} are defined as follows.

\begin{definition}
  The pair $(\eta, \mathbf{q})$ with $\eta=\eta(\nabla \phi)$ is termed a \textit{variation entropy--variation entropy flux pair} if
  \begin{itemize}
   \item $\eta$ is convex
   \item $\eta$ satisfies the homogeneity property: $\nabla \phi \cdot \dfrac{\partial \eta}{\partial \nabla \phi} = \eta$
   \item $\mathbf{q}$ satisfies $\mathbf{q} = \eta \dfrac{\partial \mathbf{f}}{\partial \phi}$.
  \end{itemize}
\end{definition}

\begin{definition}
We call the function $\phi$ a \textit{variation entropy solution} if it is an integral solution of the conservation law and formally satisfies for each variation entropy--variation entropy flux pair:
\begin{align*}
 \partial_t \eta
+\nabla \cdot  \mathbf{q} \leq 0.
\end{align*}
\end{definition}

\noindent An important class of variation entropies is formed by the semi-norms.\\

\noindent \begin{theorem}
  If a function is a semi-norm then it is a variation entropy.
\end{theorem}

\noindent The famous TVD property is special case of the following corollary.\\

\noindent \begin{corollary}In case of periodic or no-inflow boundaries the variation entropy decays in time via:
  \begin{align*}
 \int_\Omega \eta(\nabla \phi(\mathbf{x},t))~{\rm d}\Omega \leq \int_\Omega \eta(\nabla \phi_0(\mathbf{x}))~{\rm d}\Omega, \quad \text{for all $t>0$}.
\end{align*}
\end{corollary}

The remainder of the paper can be summarized as follows. 
Section \ref{sec:Preliminary} provides a brief summary of the entropy solutions in the classical sense.
In Section \ref{sec:The variation entropy condition} the concept of the variation entropy solutions is presented.
This section identifies the class of possible variation entropies.
Section \ref{sec:Characterization of variation entropies} discusses the selection of of variation entropies.
In particular, the well-known TVD property is here presented in an entropy context.
Finally, Section \ref{sec:Conclusions} draws the conclusions and outline avenues for future research.

\section{Entropy solutions in the classical sense}\label{sec:Preliminary}
Let $\Omega \subset \mathbb{R}^d$ be an open connected domain. Let us consider the scalar conservation problem: find $\phi: \Omega\times\mathcal{I} \rightarrow \mathbb{R}$ such that
\begin{align}\label{eq: nonlinear conservation law}
\partial_t \phi + \nabla \cdot \mathbf{f} =& ~0,\quad \quad (\mathbf{x},t) \in \Omega \times \mathcal{I},
\end{align}
subject to the initial condition $\phi(\mathbf{x},0) = \phi_0(\mathbf{x}) \in L^{\infty}(\Omega)$. 
Here $\mathbf{f}=\mathbf{f}(\phi) \in \mathcal{C}(\Omega,\mathbb{R})$ is the (nonlinear) flux, the spatial coordinate denotes $\mathbf{x} \in \Omega$, the time is $t \in \mathcal{I}=(0,T)$ with $T>0$. 
Solutions of (\ref{eq: nonlinear conservation law}) can contain discontinuities (shocks, rarefaction waves) which motivates the search for weak solutions. 
A weak solution $\phi \in L^{\infty}(\Omega,\mathbb{R}_+)$ of (\ref{eq: nonlinear conservation law}) satisfies
\begin{align}\label{eq: nonlinear conservation law: weak sol}
  \int_{\mathbb{R}_+} \int_{\Omega} \phi \partial_t \psi + \mathbf{f}(\phi) \nabla \psi~  {\rm d}\Omega~{\rm d}t + \int_{\Omega} \phi_0 \psi_0  ~{\rm d}\Omega=0
\end{align}
for all test functions $\psi \in \mathcal{C}_c^1(\Omega,\mathbb{R}^+)$ with $\psi_0(\mathbf{x})=\psi(\mathbf{x},0)$. 
Weak solutions are generally not unique.

Let us first consider the case of smooth solutions. Let $\eta =\eta(\phi) \in \mathcal{C}^1(\mathbb{R})$ be a convex function. Multiplying (\ref{eq: nonlinear conservation law, vanishing viscosity}) with $\partial \eta/\partial \phi$ shows that $\eta$ satisfies a conservation law
\begin{align}\label{eq: classical phi smooth}
\partial_t \eta + \nabla \cdot \mathbf{q} = 0,
\end{align}
when the flux $\mathbf{q}$ satisfies the \textit{compatibility condition}:
\begin{align}\label{eq: compatibility condition}
 \dfrac{\partial \mathbf{q}}{\partial \phi} = \dfrac{\partial \eta}{\partial \phi} \dfrac{\partial \mathbf{f}}{\partial \phi}.
\end{align}

\begin{definition}\label{def: VE}(Entropy/entropy-flux pair)
  The pair $(\eta, \mathbf{q})$ is called an entropy/entropy-flux pair when $\eta$ is a convex function and $\mathbf{q}$ fulfills the compatibility condition (\ref{eq: compatibility condition}). The function $\eta$ is referred to as the \textit{entropy function} and $\mathbf{q}$ as the \textit{entropy flux}. 
\end{definition}
When discontinuities appear the chain rule cannot be applied anymore and (\ref{eq: classical phi smooth}) is replaced by an inequality:
\begin{align}\label{eq: entropy condition}
\partial_t \eta + \nabla \cdot \mathbf{q} \leq 0.
\end{align}
The entropy condition (\ref{eq: entropy condition}) tells us that the entropy $\eta$ dissipates at shock waves. 
This inequality should be understood as
\begin{align}\label{eq: classical phi nonsmooth, rigorous}
 \displaystyle\int_{0}^{\infty}\int_\Omega \eta(\phi) \partial_t w + \mathbf{q}(\phi)\cdot \nabla w ~{\rm d}\Omega ~{\rm d}t \geq 0,
\end{align}
for all $w \in \mathcal{C}_c^\infty(\Omega \times (0,\infty)), w \geq 0$.
\begin{definition}\label{def: VE}(Entropy solution)
The function $\phi$ is called an \textit{entropy solution} or \textit{entropic} if it is an integral solution and additionally satisfies (\ref{eq: classical phi nonsmooth, rigorous}) for each entropy/entropy-flux pair.
\end{definition}
Consider solutions $\phi^\epsilon: \Omega\times\mathcal{I} \rightarrow \mathbb{R}$ of the approximate viscous problem:
\begin{align}\label{eq: nonlinear conservation law, vanishing viscosity}
\partial_t \phi^\epsilon + \nabla \cdot \mathbf{f}^\epsilon =&  ~\epsilon \Delta \phi^\epsilon,\quad \quad (\mathbf{x},t) \in \Omega \times \mathcal{I},
\end{align}
with $\mathbf{f}^\epsilon=\mathbf{f}(\phi^\epsilon)$.
The vanishing viscosity solution $\phi$ is now defined as: $\phi^\epsilon \rightarrow \phi$ a.e. for $\epsilon \rightarrow 0$.

\begin{theorem}
Vanishing viscosity solutions are entropy solutions.
\end{theorem}
\begin{proof}
This is a direct consequence of the convexity of $\eta$ and the compatibility condition (\ref{eq: compatibility condition}). For details see Evans \cite{Evans02}.
\end{proof}

\begin{theorem}
Entropy solutions of scalar conservation laws are unique.
\end{theorem}
\begin{proof}
See Evans \cite{Evans02}.
\end{proof}
In case $\Omega$ is a periodic domain or has no-inflow boundaries (the inflow is characterized by $\partial \mathbf{f}/\partial \phi \cdot \mathbf{n} \leq 0$ where $\mathbf{n}$ is the outward normal), integration of (\ref{eq: entropy condition}) over $\Omega$ leads to a decay of the overall entropy:
\begin{align}
 \dfrac{\rm d}{{\rm d}t} \int_\Omega \eta(\phi(\mathbf{x},t))~{\rm d}\Omega \leq 0,
\end{align}
which implies:
\begin{align}
 \int_\Omega \eta(\phi(\mathbf{x},t))~{\rm d}\Omega \leq \int_\Omega \eta(\phi_0(\mathbf{x}))~{\rm d}\Omega,\quad \text{for all }t\geq 0.
\end{align}
Note that taking $\eta(\phi) = \phi^2/2$ leads to the usual $L^2-$stability from linear theory for hyperbolic equations.
We refer to \cite{levequebook1992} for more details.

\begin{rmk}
  We can also consider flux functions of the form $\mathbf{f} = \mathbf{f}(\phi,\nabla \phi)$.
  For the sake of simplicity we restrict ourselves to the case where the matrix $\partial \mathbf{f}/\partial \nabla \phi$ is of the form $-k \mathbf{I}$ with the scalar $k\geq 0$ and the identity matrix $\mathbf{I}$.
  The entropy condition (\ref{eq: entropy condition}) now takes the form
  \begin{align}
\partial_t \eta + \nabla \cdot \mathbf{q} -k \Delta \eta \leq 0.
\end{align}
\end{rmk}

\section{Variation entropy solutions}\label{sec:The variation entropy condition}
In this section we introduce the notion of variation entropy solutions.

\subsection{The concept}
We present the variation entropy concept for scalar conservation laws. The extension to systems of conservation laws is straightforward and we omit it here for the sake of simple notation. We consider the nonlinear conservation law
\begin{align}\label{eq: nonlinear conservation law  diffusion source}
\partial_t \phi + \nabla \cdot \mathbf{f} =&~ 0,\quad \quad (\mathbf{x},t) \in \Omega \times \mathcal{I},
\end{align}
with flux $\mathbf{f} = \mathbf{f}(\phi)$. Let us first consider smooth solutions.
The main idea is to look at the associated entropy relation of the spatial gradient (or variation) of the conservation law instead of that of the plain conservation law, i.e. consider the entropy relation of the system of equations:
\begin{align}\label{eq: grad of conservation law}
\nabla (\partial_t \phi) + \nabla (\nabla \cdot \mathbf{f}) =& ~\mathbf{0}.
\end{align}
The motivation of the approach stems from the observation that sharp layers in solution profiles are characterized by large gradients. By considering a convex function of the solution gradient these sharp layers can be identified.
In Figure \ref{fig:The concept of the variation entropy} we sketch the concept.
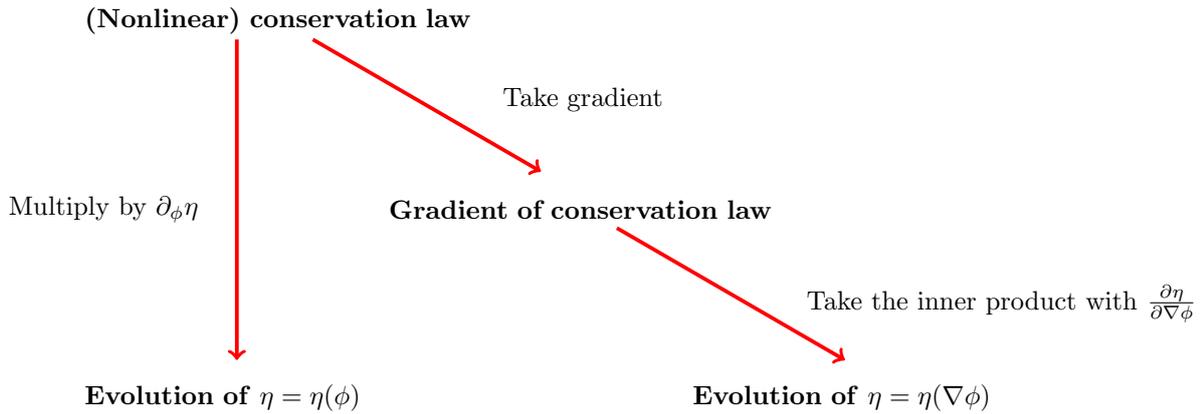
\begin{figure}[h]
  \begin{center}
\begin{tikzpicture}
\node[text width=7cm] at (-5.0,-1.75) {\textbf{(Nonlinear) conservation law}};
\node[text width=7cm] at (-5.0,-6.75) {\textbf{Evolution of} $\eta=\eta(\phi)$};
\node[text width=7cm] at (-6.0,-4.25) {Multiply by $\partial_{\phi} \eta$};
\draw[->, line width=0.5mm, red ] (-5.5,-2.00) -- (-2.5,-3.75);
\node[text width=7cm] at (0.5,-2.80) {Take gradient};
\node[text width=7cm] at (-1.0,-4.25) {\textbf{Gradient of conservation law}};
\draw[->, line width=0.5mm, red ] (-1.5,-4.50) -- (1.5,-6.25);
\node[text width=7cm] at (3.0,-6.75) {\textbf{Evolution of} $\eta=\eta(\nabla \phi)$};
\node[text width=7cm] at (4.5,-5.5) {Take the inner product with $\frac{\partial \eta}{\partial \nabla \phi}$};
\draw[->, line width=0.5mm, red ] (-6.5,-2.00) -- (-6.5,-6.25);
\end{tikzpicture}
  \end{center}
  \caption{The concept of the variation entropy for a smooth solution. In the classical approach one considers the entropy of the conservation law. The idea of the variation entropy approach is to first take the gradient of the conservation law and subsequently introducing the entropy concept.}
  \label{fig:The concept of the variation entropy}
\end{figure}

\begin{lem}\label{lem:evolution}(Evolution equation)
Let $\phi: \Omega\times\mathcal{I} \rightarrow \mathbb{R}$ be a smooth solution of the conservation law (\ref{eq: nonlinear conservation law  diffusion source}) and let $\eta:\mathbb{R}^d \rightarrow \mathbb{R}$ be a twice differentiable convex function of the gradient of $\phi$, i.e. $\eta=\eta(\nabla \phi)$. 
The temporal evolution of $\eta$ reads:
\begin{align}\label{variation entropy evolution}
\partial_t \eta
+\nabla \cdot  \mathbf{q} 
         =& \mathscr{A},
\end{align}
where the flux $\mathbf{q}$ and the non-conservative term $\mathscr{A}$ are respectively given by:
\begin{subequations}\label{variation entropy evolution terms}
  \begin{alignat}{2}
   \mathbf{q} =& \left(\frac{\partial \eta}{\partial \nabla \phi} \cdot \nabla \phi\right) \frac{\partial \mathbf{f}}{\partial \phi}, \\
 \mathscr{A}  =& \left(\mathbf{H}_{\nabla \phi} \eta \nabla \phi \right)\cdot\left(\mathbf{H}_{\mathbf{x}} \phi \frac{\partial \mathbf{f}}{\partial \phi} \right).
   \end{alignat}
\end{subequations}
The contraction is defined as $\mathbf{A}\colon \mathbf{B}={\rm Tr}(\mathbf{A}\mathbf{B}^T)$,  $(\mathbf{H}_{\mathbf{x}} \phi)_{mn}=\partial^2 \phi/\partial_{x_m}\partial_{x_n}$ is the (symmetric) Hessian of $\phi$ and $(\mathbf{H}_{\nabla \phi} \eta)_{mn}=\partial^2 \eta/\partial_{\nabla_m \phi}\partial_{\nabla_n \phi}$ is the (symmetric) Hessian of $\eta$.
\end{lem}
Under certain conditions presented in following subsections, the function $\eta$ plays the role of an entropy. To distinguish from the classical entropy setting, we will use the term \textit{variation entropy} which will be defined later.

\begin{proof}(Lemma \ref{lem:evolution})
Changing the order of differentiation and subsequently taking the inner product of (\ref{eq: grad of conservation law}) with the \textit{variation entropy variables} $\partial \eta/\partial \nabla \phi$ (vector-valued) we find:
\begin{equation}\label{eq: abs of grad of conservation law}
\partial_t \eta +  \frac{\partial \eta}{\partial \nabla \phi} \cdot \nabla (\nabla \cdot \mathbf{f})  = 0.
\end{equation}
Consider the second term of (\ref{eq: abs of grad of conservation law}) in isolation. We interchange the gradient and divergence operators and use the product rule to arrive at:
\begin{align}\label{second term}
\frac{\partial \eta}{\partial \nabla \phi}\cdot  (\nabla( \nabla \cdot \mathbf{f} )) 
=&\frac{\partial \eta}{\partial \nabla \phi} \cdot(\nabla \cdot ( \nabla \mathbf{f} )^T ) \nn \\
=&\nabla \cdot \left ( \nabla \mathbf{f} \frac{\partial \eta}{\partial \nabla \phi} \right )   - \left(\nabla \left ( \frac{\partial \eta}{\partial \nabla \phi} \right )\right)^T : \nabla  \mathbf{f}.
\end{align}
Here we use the notation $\nabla \cdot \mathbf{T} = \partial_{x_j} T_{ij}$ for the divergence of a tensor.
The first term of (\ref{second term}) yields the flux term on the left-hand side of (\ref{variation entropy evolution}). 
Using the chain rule we get straightforwardly
\begin{align}\label{first term of second term}
\nabla \mathbf{f} \frac{\partial \eta}{\partial \nabla \phi}
=& \left(\frac{\partial \eta}{\partial \nabla \phi}  \cdot  \nabla \phi\right) \frac{\partial \mathbf{f}}{\partial \phi} . 
\end{align}
Next, by using the identity
\begin{align}
   \left(\nabla \left ( \frac{\partial \eta}{\partial \nabla \phi} \right )\right)^T=& \mathbf{H}_{\mathbf{x}} \phi  \mathbf{H}_{\eta},
  \end{align}
the second term of (\ref{second term}) can be written as
\begin{align}\label{second term of second term}
 \left(\nabla \left ( \frac{\partial \eta}{\partial \nabla \phi} \right )\right)^T : \nabla  \mathbf{f} = \left(\mathbf{H}_{\nabla \phi} \eta \nabla \phi \right)\cdot\left(\mathbf{H}_{\mathbf{x}} \phi \frac{\partial \mathbf{f}}{\partial \phi} \right).
  \end{align}
Combining (\ref{eq: abs of grad of conservation law}), (\ref{second term}), (\ref{first term of second term}) and (\ref{second term of second term}) leads to the claim.
 
\end{proof}

The flux term on the left-hand side of (\ref{variation entropy evolution}), composed of a convective component, redistributes $\eta$ over the domain.
In absence of the non-conservative term $\mathscr{A}$ the temporal evolution would mirror the classical entropy case: it would satisfy a conservation law.
We proceed with identifying the class of functions $\eta$ that closely resembles the classical entropy case.

\subsection{Characterization of variation entropies}
To closely match the classical entropy concept, the variation entropy should equip solutions with a dissipative condition. As such the influence of the advection term $\mathscr{A}$ on the right-hand side of (\ref{variation entropy evolution}) is unwanted. Thus it cannot be part of the variation entropy concept. As a second property, we copy the convexity demand of $\eta$ from the classical case. Furthermore, since the concept serves to identify sharp layers, we demand that $\eta$ vanishes in absence of spatial gradients. We propose the following design condition.\\

\noindent \textit{Design condition:}\\

\noindent We seek a function $\eta$ such that:
\begin{itemize}
 \item $\mathscr{A} = 0$ (dissipative condition),
 \item $\eta$ is convex,
 \item $\eta(0) = 0$.
\end{itemize}
The following lemma identifies the class of functions $\eta$ that satisfy the design condition.
\begin{lem}\label{thm advection}(Variation entropy design condition)
A convex function $\eta$ satisfies the design condition for a general conservation law if and only if it fulfills the homogeneity property:
  \begin{align}\label{eq: homogeneity property}
    \nabla \phi \cdot \dfrac{\partial \eta}{\partial \nabla \phi} = \eta.
  \end{align} 
\end{lem}
\begin{proof}
The advection term $\mathscr{A}$ vanishes for a general conservation law if and only if the vectors $\mathbf{H}_{\nabla \phi} \eta \nabla \phi$ and $\mathbf{H}_{\mathbf{x}} \phi \partial \mathbf{f} /\partial \phi$ are orthogonal. 
Since the sign of each of the components of the advective speed $\partial \mathbf{f} /\partial \phi$ is undetermined, the entries of $\mathbf{H}_{\mathbf{x}} \phi \partial \mathbf{f} /\partial \phi$ can be positive or negative. 
Thus the only problem-independent vector orthogonal to $\mathbf{H}_{\mathbf{x}} \phi \partial \mathbf{f} /\partial \phi$ is the zero vector. 
Hence seeking a function $\eta$ for which $\mathscr{A}$ vanishes for general conservation law is equivalent to searching a function $\eta$ for which
\begin{align}
 \mathbf{H}_{\nabla \phi} \eta \nabla \phi = 0.
\end{align}
Observe that this system of equations can be cast into the form:
  \begin{align}
    \dfrac{\partial}{\partial \nabla \phi}\left(\nabla \phi\cdot \dfrac{\partial \eta}{\partial \nabla \phi}\right) = \dfrac{\partial \eta}{\partial \nabla \phi}.
  \end{align}
  Integration of the $i$-th equation with respect to $\nabla_i \phi$ provides
  \begin{align}\label{eq: theorem, integration}
    \nabla \phi \cdot  \dfrac{\partial \eta}{\partial \nabla \phi} = \eta + c, 
  \end{align}
  where $c$ is a constant. The condition $\eta(0) = 0$ leads to $c=0$, and thus the homogeneity property (\ref{eq: homogeneity property}) follows.
\end{proof}
We call a function that satisfies the design condition a \textit{variation entropy}, as stated in the following definition.
\begin{definition}\label{def VE}(Variation entropy)
  A function $\eta=\eta(\nabla \phi)$ is termed a \textit{variation entropy} if
  \begin{itemize}
   \item $\eta$ is convex
   \item $\eta$ satisfies the homogeneity property (\ref{eq: homogeneity property}).
  \end{itemize}
\end{definition}

An alternative form of the homogeneity property is stated in the following lemma.
\begin{lem}\label{lemma:Homogeneous function}(Homogeneous function)
   A variation entropy is a positive homogeneous function of degree $1$:
     \begin{align}\label{eq: pos homogeneous function degree 1}
   \eta(\alpha \mathbf{v}) = \alpha \eta(\mathbf{v}),\quad \text{for all } \mathbf{v} \in \mathbb{R}^d, \alpha \geq 0.
  \end{align}
\end{lem}
\begin{proof}
This is a direct consequence of Euler's homogeneous function theorem.
\end{proof}
By employing Lemma \ref{lemma:Homogeneous function}, convexity is equivalent to sub-additivity.
\begin{lem}\label{lem: sub-additivity}(Sub-additivity)
 A positive homogeneous function of degree $1$ is convex if and only if it is sub-additive:
 \begin{align}
  \eta (\mathbf{v}_1+\mathbf{v}_2) \leq \eta(\mathbf{v}_1) + \eta(\mathbf{v}_2),\quad \text{for all }\mathbf{v}_1,\mathbf{v}_2 \in \mathbb{R}^d.
 \end{align}
\end{lem}

Combining Lemma \ref{lemma:Homogeneous function} and \ref{lem: sub-additivity} leads to the following theorem.
\begin{theorem}(Variation entropy)
  A function $\eta=\eta(\nabla \phi)$ is a \textit{variation entropy} if and only if
  \begin{itemize}
   \item $\eta$ is positive homogeneous function of degree $1$
   \item $\eta$ is sub-additive.
  \end{itemize}
\end{theorem}

\begin{theorem}\label{thm:semi-norm}(Semi-norm)
  If a function is a semi-norm then it is a variation entropy.
\end{theorem}
\begin{proof}
  The axioms of a seminorm are absolute homogeneity, i.e.
  \begin{align}\label{eq: abs homogeneous function}
    \eta(\alpha \mathbf{v}) = |\alpha| \eta(\mathbf{v}),\quad \text{for all } \mathbf{v} \in \mathbb{R}^d, \alpha \in \mathbb{R},
  \end{align}
  and sub-additivity. The absolute homogeneity demand (\ref{eq: abs homogeneous function}) is a specific case of the positive homogeneity (\ref{eq: pos homogeneous function degree 1}).
\end{proof}

\begin{theorem}\label{thm:linear function}(Linear function)
  If a function is linear then it is a variation entropy.
\end{theorem}

Note that theorems \ref{thm:semi-norm} and \ref{thm:linear function} provide sufficient but not necessary conditions.
We remark that a linear combinations of variation entropies form again a variation entropy.
\begin{corollary}\label{corollary: lin comb}(Linear combination)
Let $\eta_k$ be variation entropies for $k=1,..,n$ for some integer $n$. 
The linear combination $\eta:=\sum_k \alpha_k \eta_k$ with $\alpha_k \in \mathbb{R}^+$ is a variation entropy.
\end{corollary}

\begin{lem}\label{thm alterative form VE}(Alternative form of the variation entropy)
A convex function $\eta=\eta(\nabla \phi)$ is a \textit{variation entropy} if and only if it is of the form
\begin{align}
 \eta=\eta(\nabla \phi) = \hat{\eta}(r,\boldsymbol{\theta}) = F(\boldsymbol{\theta})r,
\end{align}
where $F=F(\boldsymbol{\theta})$ is a scalar-valued function and the spherical coordinates are the radius $r$ and the angles $\boldsymbol{\theta}$ corresponding to the coordinates $\nabla \phi$. 
The convexity demand translates in the two-dimensional case to
\begin{align}
 F(\theta) + F''(\theta)  \geq 0.
\end{align}
\end{lem}
The restrictions on the function $F$ in the $3$-dimensional case are more involved. 
We refer the interested reader to \ref{appendix: 3-dimensional case} for the details.

\begin{rmk}
Note that $\hat{\eta}=r F(\boldsymbol{\theta})$ is not differentiable in the origin.
However the origin is an important part of the domain since this is where $\phi$ attains a local extremum.
A possible way to bypass non-differentiability is to regularize $\eta$ and thus to work with an approximate variation entropy.
\end{rmk}
\begin{proof}(Lemma \ref{thm alterative form VE})
We start from the homogeneity property
\begin{align}\label{eq: theorem, integration2}
   \mathbf{v} \cdot \dfrac{\partial \eta}{\partial \mathbf{v}} = \eta, \quad \text{ for all } \mathbf{v} \in \mathbb{R}^d.
\end{align}  
  We now switch to spherical coordinates, i.e. the coordinates consist of a radial coordinate $r$ and $d-1$ angular coordinates $\theta \in [0, 2\pi), \varphi_m \in [0, \pi], m = 1,\dots, d-2$. 
  The transformation is given by:
  \begin{align}\label{eq: transformation spherical coordinates}
    v_1 =&  r \cos \theta \prod_{l=1}^{d-2} \sin \varphi_l, \nn\\
    v_2 =&  r \sin \theta \prod_{l=1}^{d-2} \sin \varphi_l, \nn \\
    v_m =&  r \cos \varphi_{m-2}  \prod_{l=m-1}^{d-2} \sin \varphi_l  \quad \text{ for }m=3,..., d-1, \nn \\
    v_d =& r \cos \varphi_{d-2}.
    \end{align}
  For $d = 3$ the third line drops out, and for $d=2$ both the third and last lines vanish.
  Both cases reduce to the well-known transformations.
  Consult \cite{blumenson1960derivation} for a derivation of a similar form.
  The direct consequence $r \partial \mathbf{v}/\partial r = \mathbf{v}$ provides
  \begin{align}
    r \dfrac{\partial \hat{\eta}}{\partial r}= r \dfrac{\partial \eta}{\partial \mathbf{v}} \cdot \dfrac{\partial \mathbf{v}}{\partial r} = \mathbf{v} \cdot\dfrac{\partial \eta}{\partial \mathbf{v}}.
  \end{align}
  This allows us to cast (\ref{eq: theorem, integration}) into the differential equation
  \begin{align}\label{eq: theorem, integration in r}
    r \dfrac{\partial \hat{\eta}}{\partial r} = \hat{\eta}.
  \end{align}
   The corresponding solution follows straightforwardly
  \begin{align}\label{eq: theorem, integration in r}
   \hat{\eta}=\hat{\eta}(r, \varphi_1, ..., \varphi_{d-2},\theta) = \hat{\eta}(r, \boldsymbol{\theta}) = F(\boldsymbol{\theta}) r,
  \end{align}
  with $F$ a scalar-valued function.\\
  
We characterize the convexity of $\eta$ by the positivity of the eigenvalues of the Hessian.
In the two-dimensional case we have $F=F(\theta)$. 
The Hessian in polar coordinates takes the form:
\begin{align}
 \mathbf{H}_{\nabla \phi} \eta =  \dfrac{F(\theta)+F''(\theta)}{r}
           \begin{pmatrix} 
             \sin^2 \theta & -\cos \theta \sin \theta \\
             -\cos \theta \sin \theta & \cos^2 \theta
           \end{pmatrix}.
\end{align}
Note that this is in line with the first demand. The eigenvalues $\lambda_1, \lambda_2$ of $\mathbf{H}_{\nabla \phi} \eta$ are
\begin{align}
 \lambda_1 =&~ 0,\\
 \lambda_2 =&~ \dfrac{F(\theta)+F''(\theta)}{r},
\end{align}
and convexity of $\eta$ follows when $F(\theta) + F''(\theta) \geq 0$.\\
\end{proof}

\subsection{Variation entropy-variation entropy flux pairs}
Employing the homogeneity property (\ref{eq: homogeneity property}) the evolution equation of a variation entropy takes the following form.
\begin{lem}\label{lem:evolution variation entropy}(Evolution equation variation entropy)
Let $\phi: \Omega\times\mathcal{I} \rightarrow \mathbb{R}$ be a smooth solution of the conservation law (\ref{eq: nonlinear conservation law  diffusion source}) and let $\eta:\mathbb{R}^d \rightarrow \mathbb{R}$ be a twice differentiable convex function of the gradient of $\phi$, i.e. $\eta=\eta(\nabla \phi)$. 
The temporal evolution of a variation entropy $\eta$ reads:
\begin{align}\label{variation entropy evolution 3}
\partial_t \eta
+\nabla \cdot  \mathbf{q} 
         =&~ 0,
\end{align}
where the flux $\mathbf{q}$ is given by:
\begin{align}\label{variation entropy evolution terms 3}
   \mathbf{q} = \eta \frac{\partial \mathbf{f}}{\partial \phi}.
\end{align}
\end{lem}
\begin{definition}\label{def: VE}(Variation entropy--variation entropy flux pair)
  The pair $(\eta, \mathbf{q})$ is called a variation entropy--variation entropy flux pair provided
  \begin{itemize}
   \item $\eta$ is a variation entropy
   \item $\mathbf{q}$ satisfies the compatibility condition $\mathbf{q} = \eta \dfrac{\partial \mathbf{f}}{\partial \phi}$.
  \end{itemize}
\end{definition}
\begin{rmk}
 Taking the derivative of the compatibility condition with respect to $\nabla \phi$ yields for $\eta=\eta(\nabla \phi)$:
 \begin{align}
  \dfrac{\partial \mathbf{q}}{\partial \nabla \phi} = \dfrac{\partial \eta}{\partial \nabla \phi} \dfrac{\partial \mathbf{f}}{\partial \phi}.
 \end{align}
 This form highlights the relation with the compatibility condition of the classical case (\ref{eq: compatibility condition}) which states for $\eta=\eta(\phi)$:
 \begin{align}
  \dfrac{\partial \mathbf{q}}{\partial \phi} = \dfrac{\partial \eta}{\partial \phi} \dfrac{\partial \mathbf{f}}{\partial \phi}.
 \end{align} 
\end{rmk}

We proceed with presenting variation entropy solutions.
In practice hyperbolic problems with discontinuities may be interesting.
When discontinuities appear the chain rule cannot be applied anymore and thus the foregoing derivations are not
valid anymore.
In the classical entropy case at this point the conservation law of the entropy is replaced by an entropy inequality.
This is well-defined at discontinuities.
For the variation entropy concept, this is not possible at a discontinuity due to the gradient in $\eta(\nabla \phi)$.
As such the concept of a variation entropy solution can be applied to problems with sharp layers but is not defined in case of discontinuities.
We wish to emphasize that variation entropy solutions may be very suitable in numerics though.
A numerical method often approximates gradients at discontinuities (e.g. continuous finite-elements) and as such ill-defined gradients do not appear.
The variation entropy solution provides a local stability condition that may be used to eliminate the spurious oscillations near sharp layers or shocks.

In the following we proceed with a smooth solution to bypass this singularity problem.
We aim to closely resemble the classical entropy case in the following and accordingly replace the variation entropy evolution equation by the \textit{variation entropy condition}:
\begin{align}\label{variation entropy cond}
\partial_t \eta
+\nabla \cdot  \mathbf{q} 
         \leq 0.
\end{align}
This means that the variation entropy evolves according to the flux $\mathbf{q}$ but may experience sudden increases.
The rigorous counterpart of (\ref{variation entropy cond}) is:
\begin{align}\label{eq: rigorous VE}
\int_{0}^{\infty} \int_{\Omega} \eta\left( \partial_t w + \dfrac{\partial \mathbf{f}}{\partial \phi}\cdot \nabla w \right) {\rm d}\mathbf{x}{\rm d}t \geq 0.
\end{align}
In a similar fashion as for entropy solutions in the classical sense we define \textit{variation entropy solutions}.
\begin{definition}\label{def: VE}(Variation entropy solution)
We call the smooth function $\phi$ a \textit{variation entropy solution} if it is an integral solution and additionally satisfies (\ref{eq: rigorous VE}) for each variation entropy--variation entropy flux pair.
\end{definition}
Let us now consider the solutions $\phi^\epsilon: \Omega\times\mathcal{I} \rightarrow \mathbb{R}$ of the approximate viscous problem:
\begin{align}\label{eq: nonlinear conservation law, vanishing viscosity 2}
\partial_t \phi^\epsilon + \nabla \cdot \mathbf{f}^\epsilon =& ~ \epsilon \Delta \phi^\epsilon,\quad \quad (\mathbf{x},t) \in \Omega \times \mathcal{I},
\end{align}
with viscosity parameter $\epsilon >0$, flux $\mathbf{f}^\epsilon \equiv \mathbf{f}(\phi^\epsilon)$.
Let the vanishing viscosity solution denote $\phi = \lim_{\epsilon \rightarrow 0} \phi^\epsilon$.
\begin{theorem}\label{thm: VV solutions are VE}(Vanishing viscosity)
   The smooth function $\phi$ is a variation entropy solution.
\end{theorem}
\begin{proof}
The proof basically is the variation entropy counterpart of the classical entropy case presented in Evans \cite{Evans02}.
The vanishing viscosity solution is an integral solution just as used for the classical entropy case. 
What remains is to show that it also satisfies the variation entropy condition. Taking the gradient of (\ref{eq: nonlinear conservation law, vanishing viscosity 2}) and subsequently the inner product with $\partial \eta/\partial \nabla \phi^\epsilon$ provides
\begin{align}\label{eq: variation entropy vanishing viscosity}
\partial_t \eta^\epsilon
+\nabla \cdot \mathbf{q}^\epsilon =\epsilon \frac{\partial \eta^\epsilon}{\partial \nabla \phi^\epsilon} \cdot \nabla ( \Delta \phi^\epsilon),
\end{align}
where the superscript $\epsilon$ indicates dependence on $\nabla \phi^\epsilon$.
Next, by applying the chain rule we can evaluate the expression as:
\begin{align}
  \Delta \eta^\epsilon =	&~ \nabla \cdot (\nabla \eta^\epsilon) \nn\\
		=&~ \nabla \cdot \left(  \mathbf{H}_{\mathbf{x}} \phi^\epsilon \frac{\partial \eta^\epsilon}{\partial \nabla \phi^\epsilon} \right) \nn \\
		=&~ \left(\mathbf{H}_{\mathbf{x}} \phi^\epsilon \mathbf{H}_{\nabla \phi^\epsilon} \eta^\epsilon\right) : \mathbf{H}_{\mathbf{x}} \phi^\epsilon + \dfrac{\partial \eta^\epsilon}{\partial \nabla \phi^\epsilon} \cdot \nabla  (\Delta \phi^\epsilon).
\end{align}
Substitution yields:
\begin{align}\label{variation entropy evolution vv}
\partial_t \eta^\epsilon
+\nabla \cdot  \mathbf{q}^\epsilon 
         =&~ \epsilon \Delta \eta^\epsilon -\epsilon~ \mathbf{H}_{\mathbf{x}} \phi^\epsilon \mathbf{H}_{\nabla \phi^\epsilon} \eta^\epsilon \mathbf{H}_{\mathbf{x}} \phi^\epsilon,
\end{align}
where the flux $\mathbf{q}^\epsilon$ satisfies the compatibility condition (\ref{variation entropy evolution terms 3}) and where the superscript $\epsilon$ denotes dependence on $\phi^\epsilon$ instead of $\phi$.
By the convexity of $\eta^\epsilon$ the second term on the right-hand side\footnote{Note that its classical entropy counterpart is $\nabla \phi^\epsilon \cdot \nabla \phi^\epsilon \partial^2 \eta^\epsilon/\partial (\phi^\epsilon)^2$.} has the sign:
\begin{align}\label{eq: vanishing viscosity limit positive}
  \left(\mathbf{H}_{\mathbf{x}} \phi^\epsilon \mathbf{H}_{\nabla \phi^\epsilon} \eta^\epsilon\right) : \mathbf{H}_{\mathbf{x}} \phi^\epsilon \geq 0,
\end{align}
which implies
\begin{align}\label{variation entropy evolution vvv}
\partial_t \eta^\epsilon
+\nabla \cdot  \mathbf{q}^\epsilon 
         \leq &~ \epsilon \Delta \eta^\epsilon.
\end{align}
We proceed by multiplying (\ref{variation entropy evolution vv}) by $w \in \mathcal{C}_c^{\infty}(\Omega \times (0,\infty)), w \geq 0$, integrate and use (\ref{eq: vanishing viscosity limit positive}):
\begin{align}
\int_{0}^{\infty} \int_{\Omega} \eta\left( \partial_t w + \dfrac{\partial \mathbf{f}}{\partial \phi}\cdot \nabla w \right) {\rm d}\mathbf{x}{\rm d}t =&~ \int_{0}^{\infty} \int_{\Omega}  \epsilon\left(\mathbf{H}_{\mathbf{x}}\phi \mathbf{H}_{\nabla \phi}\eta\right) : \mathbf{H}_{\mathbf{x}}\phi w-  \epsilon \eta \Delta w {\rm d}\mathbf{x}{\rm d}t  \nn\\
          \geq&~ - \int_{0}^{\infty} \int_{\Omega} \epsilon \eta \Delta w {\rm d}\mathbf{x}{\rm d}t .
\end{align}
Applying the dominated convergence theorem we get
\begin{align}\label{eq: rigorous VE2}
\int_{0}^{\infty} \int_{\Omega} \eta\left( \partial_t w + \dfrac{\partial \mathbf{f}}{\partial \phi}\cdot \nabla w \right) {\rm d}\mathbf{x}{\rm d}t \geq 0.
\end{align}
\end{proof}
\begin{corollary}\label{cor: total} (Decay of variation entropy)
  Let $\Omega$ be a periodic domain or have no-inflow boundaries and let $\phi$ be a smooth function.
  The variation entropy of a vanishing viscosity solution decays in time:
  \begin{align}
 \int_\Omega \eta(\nabla \phi(\mathbf{x},t))~{\rm d}\Omega \leq \int_\Omega \eta(\nabla \phi_0(\mathbf{x}))~{\rm d}\Omega, \quad \text{for all $t>0$}.
\end{align}
\end{corollary}
\begin{proof}
Integration of (\ref{variation entropy cond}) yields:
\begin{align}\label{eq: TVD constraint eta}
 \dfrac{{\rm d}}{{\rm d}t}\int_\Omega \eta(\nabla \phi(\mathbf{x},t)) ~{\rm d}\Omega \leq 0\quad  \Rightarrow \quad \int_\Omega \eta(\nabla \phi(\mathbf{x},t))~{\rm d}\Omega \leq \int_\Omega \eta(\nabla \phi_0(\mathbf{x}))~{\rm d}\Omega,
\end{align}
for all $t>0$.
\end{proof}

\subsection{Augmented conservation laws}
We consider `augmented conservation laws', i.e. PDEs with a convective, diffusive
and source component of the form: 
 \begin{align}\label{eq: nonlinear conservation law diffusion source aug}
\partial_t \phi + \nabla \cdot \mathbf{f} =& ~s,\quad \quad (\mathbf{x},t) \in \Omega \times \mathcal{I},
\end{align}
with flux $\mathbf{f}=\mathbf{f}(\phi, \nabla \phi)$ and source term $s=s(\phi)$. Here we assume that the matrix $\partial \mathbf{f}/\partial \nabla \phi$ is negative semi-definite.
Let $(\eta, \mathbf{q})$ be a variation entropy/variation entropy flux pair.
A straightforward computation shows that using the homogeneity property (\ref{eq: homogeneity property}), the evolution equation takes the form:
\begin{align}\label{variation entropy evolution aug}
\partial_t \eta
+\nabla \cdot  \mathbf{q} 
         =& \mathscr{D}+\mathscr{S},
\end{align}
where the flux $\mathbf{q}$ and the non-conservative terms $\mathscr{D}$ and $\mathscr{S}$ are respectively given by:
\begin{subequations}\label{variation entropy evolution terms}
  \begin{alignat}{2}
   \mathbf{q} =& \eta \frac{\partial \mathbf{f}}{\partial \phi} + \frac{\partial \mathbf{f}}{\partial \nabla \phi}  \nabla \eta, \\
 \mathscr{D}  =& \left(\mathbf{H}_{\mathbf{x}} \phi \mathbf{H}_{\nabla \phi} \eta \right): \left(\frac{\partial \mathbf{f}}{\partial \nabla \phi} \mathbf{H}_{\mathbf{x}} \phi \right),\\
 \mathscr{S}  =& \frac{\partial s}{\partial \phi} \eta.
  \end{alignat}
\end{subequations}

We emphasize that this form closely resembles an augmented conservation law with convection, diffusion and reaction components.
The reaction term is the only term that can create variation entropy (remark that it vanishes for a constant source $s$).
Note that negative eigenvalues of the diffusion matrix are an essential requirement for well-posedness. 
Positive eigenvalues create variation entropy which leads to a blow-up of the solutions.
In the next lemma we show that the term $\mathscr{D}$ on the right-hand side of (\ref{variation entropy evolution aug}) destroys variation entropy.

\begin{lem}\label{prop diffusion}(Negativity of $\mathscr{D}$) The term $\mathscr{D}$
on the right-hand side of (\ref{variation entropy evolution aug}) contributes to dissipation of the variation entropy, i.e. it takes negative values only.
\end{lem}
\begin{proof}
The convexity of the function $\eta= \eta(\nabla \phi)$ implies that there exists a real nonsingular matrix $\mathbf{M}$ such that
\begin{align}\label{B is PD}
   \mathbf{H}_{\nabla \phi} \eta = \mathbf{M}^T\mathbf{M}.
\end{align}
Substitution into the expression for $\mathscr{D}$ gives
\begin{align}
 \mathscr{D} = \left(\mathbf{H}_{\mathbf{x}} \phi \mathbf{H}_{\nabla \phi} \eta \right): \left(\frac{\partial \mathbf{f}}{\partial \nabla \phi} \mathbf{H}_{\mathbf{x}} \phi \right) = \left(\mathbf{M} \mathbf{H}_{\mathbf{x}} \phi  \frac{\partial \mathbf{f}}{\partial \nabla \phi}\right) : \left(\mathbf{M} \mathbf{H}_{\mathbf{x}} \phi  \right)  \leq 0,
\end{align}
where the inequality follows via the fact that the diffusivity matrix is negative semi-definite.
\end{proof}

\section{Selection of the variation entropy}\label{sec:Characterization of variation entropies}
Here we provide some examples of variation entropies and discuss objectivity of variation entropies. We present regularization of the $2$-norm in detail. For the sake of simplicity we assume smooth solutions in this section.

\subsection{Some examples}
Some examples of variation entropies are:
\begin{itemize}
 \item The linear function $\eta(\nabla \phi) = \mathbf{a} \cdot \nabla \phi$, with $\mathbf{a} \in \mathbb{R}^d$.
\item The function $\eta(\nabla \phi) = \vertiii{\nabla \phi}_{\mathbf{A}}$ defined as $\vertiii{\nabla \phi}_{\mathbf{A}}^2 := \nabla \phi^T \mathbf{A} \nabla \phi $, with $\mathbf{A} \in \mathbb{R}^{d\times d}$ a symmetric positive-semidefinite matrix.
\item The standard $p$-norm, i.e. $\eta(\nabla \phi)=\|\nabla \phi \|_p$ with $p \geq 1$.
\end{itemize}
By Corollary \ref{corollary: lin comb}, any superposition of the previous variation entropies is also a variation entropy.

\begin{prop}\label{eq:form F}
The alternative form of the variation entropy is
\begin{align}
  \hat{\eta}(r,\boldsymbol{\theta}) = r F(\boldsymbol{\theta})
\end{align}
where $F=F(\boldsymbol{\theta})$ is given by
\begin{align}
  \text{ for }d=2: \quad F(\theta) = \left\{
  \begin{array}{c l} a_1 \cos \theta + a_2 \sin \theta           						   & \text{ for }\eta(\nabla \phi) = \mathbf{a}\cdot \nabla \phi \\[8pt]
  \left( a_{11} \cos^2 \theta  + 2a_{21}\cos\theta\sin\theta +  a_{22}\sin^2 \theta\right)^{1/2} & \text{ for }\eta(\nabla \phi) = \vertiii{\nabla \phi}_{\mathbf{A}}\\[8pt]
         ( |\cos \theta|^p + |\sin \theta|^p)^{1/p} 						   & \text{ for }\eta(\nabla \phi) = \|\nabla \phi\|_p
  \end{array}
  \right.
\end{align}
\begin{align}
  \text{ for }d=3: \quad F(\theta,\varphi) = \left\{
  \begin{array}{c l} a_1 \cos \theta \sin \varphi+ a_2 \sin \theta \sin \varphi + a_3 \cos \varphi       	   & \text{ for }\eta(\nabla \phi) = \mathbf{a}\cdot \nabla \phi \\[8pt]
  \left( a_{11} \cos^2 \theta \sin^2 \varphi   + a_{22}\sin^2\theta\sin^2\varphi +a_{33} \cos^2 \varphi \right. & \nn\\
  \left.+ 2a_{21}\cos\theta\sin\theta\sin^2\varphi +2a_{31} \cos \theta \sin \varphi \cos \varphi\right. & \nn\\
  \left.+ 2a_{32}\sin\theta\sin\varphi\cos\varphi\right)^{1/2} & \text{ for }\eta(\nabla \phi) = \vertiii{\nabla \phi}_{\mathbf{A}}\\[8pt]
        ( |\cos \theta\sin \varphi|^p + |\sin \theta\sin \varphi|^p+ |\cos \varphi|^p)^{1/p}								   & \text{ for }\eta(\nabla \phi) = \|\nabla \phi\|_p.
  \end{array}
  \right.
\end{align}
\end{prop}
\begin{proof}
 We present the proof only for $\eta(\nabla \phi) = \vertiii{\nabla \phi}_{\mathbf{A}}$ with $d=2$, the other forms follow similarly.
 A direct calculation provides
   \begin{align}
     \nabla \phi^T \mathbf{A}\nabla \phi &= a_{11} \partial_x \phi^2  +
        2a_{21}\partial_x \phi \partial_y \phi + a_{22}\partial_y \phi^2.      
   \end{align}
   Next, we can trivially write
      \begin{align}\label{eq: local gradient cartesian coordinates}
    \vertiii{\nabla \phi }_{\mathbf{A}} &= \|\nabla \phi \|_2\left( a_{11} \frac{\partial_x \phi^2}{\|\nabla \phi \|_2^2}  +
      2a_{21}\frac{\partial_x \phi \partial_y \phi}{\|\nabla \phi \|_2^2} + a_{22}\frac{\partial_y \phi^2}{\|\nabla \phi \|_2^2}\right)^{1/2}.
   \end{align}
   Switching to polar coordinates with radial distance $r$ and angle $\theta$, the right-hand side of (\ref{eq: local gradient cartesian coordinates}) takes the form
   \begin{align}
     r F(\theta) = r \left( a_{11} \cos^2 \theta  +
         2a_{21}\cos\theta\sin\theta +  a_{22}\sin^2 \theta\right)^{1/2}.
   \end{align}
   Note that fulfilling the convexity condition requires the positive semi-definiteness of the matrix $\mathbf{A}$:
   \begin{align}\label{eq:convex condition SPD}
     F(\theta) + F''(\theta) &=   \dfrac{\det \mathbf{A}}{F(\theta)^3} \geq 0.
   \end{align}
   Here $\det \mathbf{A}$ denotes the determinant of matrix $\mathbf{A}$.
\end{proof}

The linear variation entropy fulfills the convexity condition in spherical coordinates with equality, for $d=2$ this reads
\begin{align}
  F(\theta) + F''(\theta) = 0.
\end{align}
Thus there is no space for the variation entropy to decrease.
It is known that the entropy does not remain constant at shocks.
This makes the linear variation entropy a non-suitable mechanism to deal with shock waves in a numerical simulation.

For the $1$-norm variation entropy, i.e. $\eta=\|\nabla \phi \|_1$ we have
\begin{align}
  F(\theta) + F''(\theta) = 0, \quad\quad \theta \neq m \pi/2\quad m \in \mathbb{Z}.
\end{align}
We remark that the $1$-norm is not differentiable along the axis.
Thus also for the $1$-norm there is also no space for the variation entropy to decrease.
For $p>1$ the convexity condition is not an equality which makes it suitable for shocks.

Next we consider the quadratic form.
Proposition \ref{eq:form F} reveals that the matrix $\mathbf{A} \in \mathbb{R}^{d\times d}$ in the quadratic form may depend on $\boldsymbol{\theta}$ (the spherical coordinate angles).
We explicitly state the evolution equation of $\eta = \vertiii{\nabla \phi}_{\mathbf{A}}$ for an augmented conservation law.
Note that the derivatives take the form
\begin{subequations}\label{eq: quadratic form derivatives}
  \begin{alignat}{2}
   \frac{\partial \eta}{\partial \nabla \phi} &= \frac{1}{\eta}\mathbf{A}  \nabla \phi,\\
   (\mathbf{H}_{\nabla \phi} \eta) &= \frac{1}{\eta} \mathbf{A}  - \frac{1}{\eta^3}\mathbf{A}  \nabla \phi (\mathbf{A} \nabla \phi)^T.
  \end{alignat}
\end{subequations}
Substitution of (\ref{eq: quadratic form derivatives}) into (\ref{variation entropy evolution aug})-(\ref{variation entropy evolution terms}) reveals that the variation entropy evolves as:
\begin{align}
\partial_t \eta
+\nabla \cdot  \mathbf{q} 
         =& \mathscr{D}+\mathscr{S},
\end{align}
where the flux $\mathbf{q}$ and the non-conservative terms $\mathscr{D}$ and $\mathscr{S}$ are respectively given by:
\begin{subequations}
  \begin{alignat}{2}
   \mathbf{q} =& \eta \frac{\partial \mathbf{f}}{\partial \phi}   +  \left(\frac{\partial \mathbf{f}}{\partial \nabla \phi}  \mathbf{H}_{\mathbf{x}} \phi\right) \frac{\nabla \phi^T \mathbf{A}}{\eta}, \\
   \mathscr{D}  =& \left(\mathbf{H}_{\mathbf{x}} \phi \left(\frac{\mathbf{A}}{\eta} -\frac{1}{\eta^3}\mathbf{A}  \nabla \phi : \mathbf{A} \nabla \phi \right) \right): \left(\frac{\partial \mathbf{f}}{\partial \nabla \phi} \mathbf{H}_{\mathbf{x}} \phi \right),\\
   \mathscr{S}  =& \frac{\partial s}{\partial \phi}\eta. 
  \end{alignat}
\end{subequations}


\subsection{Objectivity}
A reasonable demand on the continuous level is to ask for objectivity (frame-invariance) of the variation entropy.
Here we solely focus on rotation invariance, as invariance by translation is immediate.
Thus we wish to identify those variation entropies which are not affected by a rotation of the coordinate system.
We denote with $\mathbf{x}$ the original spatial coordinates and $\tilde{\mathbf{x}} = \mathbf{R} \mathbf{x}$ the rotated coordinate system with rotation matrix $\mathbf{R}$ (i.e. $\mathbf{R}^T = \mathbf{R}^{-1}$).

Let us first consider variation entropies which depend on spatial coordinates solely via $\nabla \phi$. This means that possible coefficients, like in $\mathbf{a}$ and $\mathbf{A}$ appearing in $\eta=\mathbf{a}\cdot \nabla \phi$ and $\eta=\vertiii{\nabla \phi}_\mathbf{A}$ respectively, are constant with respect to the spatial coordinates. In this case rotation invariance may be written as:
\begin{align}
 \eta (\nabla_{\tilde{\mathbf{x}}} \phi) = \eta (\nabla_{\mathbf{x}} \phi),
\end{align}
or equivalently:
\begin{align}
 \eta (\mathbf{R}\nabla_{\tilde{\mathbf{x}}} \phi) = \eta (\nabla_{\tilde{\mathbf{x}}} \phi),\quad \quad \eta (\mathbf{R}^T\nabla_{\mathbf{x}} \phi) = \eta (\nabla_{\mathbf{x}} \phi).
\end{align}
The subscript refers to the corresponding coordinate system.
\begin{theorem}(Objectivity)
 The only objective variation entropy with constant coefficients is the total variation measured in the $2$-norm, $\eta=\|\nabla \phi\|_2$ (up to multiplication by a constant).
\end{theorem}

\begin{proof}
 We use the alternative form of Lemma \ref{thm alterative form VE}. Let the angles of the rotation matrix $\mathbf{R}$ be $\boldsymbol{\varrho}$ and let polar angles of $\nabla \phi$ denote $\boldsymbol{\theta}$. A direct computation results in
 \begin{align}
  \eta (\mathbf{R}\nabla \phi) = r F(\boldsymbol{\theta}+\boldsymbol{\varrho}) =\hat{\eta} (r,\boldsymbol{\theta}+\boldsymbol{\varrho}).
 \end{align}
 Demanding $\eta(\mathbf{R}\nabla \phi) = \eta(\nabla \phi)$ provides that $F(\boldsymbol{\theta}+\boldsymbol{\varrho})=F(\boldsymbol{\theta})$ for all angles $\boldsymbol{\varrho}$, i.e. $F$ is a constant.
\end{proof}

Thus the only objective $p$-norm with constant coefficients is the norm with $p=2$. In particular, we wish to emphasize that \textit{the $1$-norm is not objective}. This makes it unsuitable for usage in non-Cartesian grid computations.

Next we proceed with the case in which the coefficients may depend on spatial coordinates. We first consider the linear variation entropy. Objectivity results when
\begin{align}
  \mathbf{a}(\mathbf{x})\cdot \nabla_{\mathbf{x}} \phi = \mathbf{a}(\tilde{\mathbf{x}})\cdot \nabla_{\tilde{\mathbf{x}}}\phi.
\end{align}
By applying the chain rule we get
\begin{align}
  \mathbf{a}(\mathbf{x})\cdot \nabla_{\mathbf{x}} \phi = \mathbf{a}(\mathbf{Rx})\cdot \mathbf{R}^T \nabla_{\mathbf{x}}\phi.
\end{align}
Thus objectivity follows when the coefficient vector satisfies:
\begin{align}\label{condition a}
  \mathbf{a}(\mathbf{R}\mathbf{x}) = \mathbf{R} \mathbf{a}(\mathbf{x}).
\end{align}
Note that when $\mathbf{a}$ represents a convection velocity, a rotation of coordinate system naturally implies the same rotation of the convective velocity.

Consider now the variation entropy $\eta = \vertiii{\nabla \phi}_{\mathbf{A}}$ with $\mathbf{A}$ a symmetric positive semi-definite matrix. This variation entropy is objective when
\begin{align}
  \vertiii{\nabla_{\mathbf{x}}\phi}_{\mathbf{A}(\mathbf{x})} = \vertiii{\nabla_{\tilde{\mathbf{x}}}\phi}_{\mathbf{A}(\tilde{\mathbf{x}})}.
\end{align}
A similar argument leads to the constraint:
\begin{align}\label{condition A}
 \mathbf{A}(\mathbf{Rx})= \mathbf{R}\mathbf{A}(\mathbf{x})\mathbf{R}^{T}.
\end{align}

Remark that $\mathbf{A}=\mathbf{I}$ is included in this case.

An overview of the variation entropy results is presented in Figure \ref{fig:tvd entropy}.

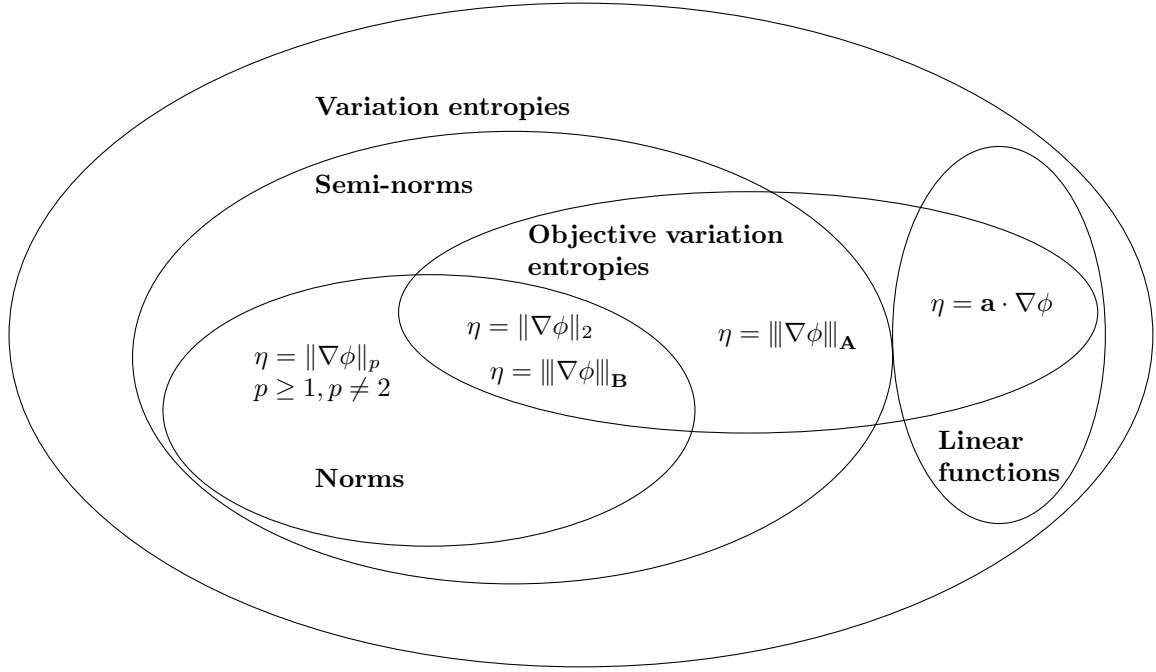
\begin{figure}[h]
  \begin{center}
\begin{tikzpicture}
\draw (2.7,0.3) ellipse (4.6cm and 1.6cm);
\draw (0.5,0) ellipse (7.525cm and 4.4cm);
\draw (-0.4,-0.3) ellipse (5.0cm and 3.0cm);
\draw (-1.5,-1.0) ellipse (3.5cm and 1.8cm);
\draw (6.0,0.0) ellipse (1.4cm and 2.5cm);
\node[text width=4cm] at (1.8,1.1) {\textbf{Objective variation entropies}};
\node[text width=3cm] at (0.5,0.1) {$\eta= \|\nabla \phi \|_2$};
\node[text width=7cm] at (0.5,3.0) {\textbf{Variation entropies}};
\node[text width=7cm] at (0.5,2.0) {\textbf{Semi-norms}};
\node[text width=2cm] at (6.2,-1.6) {\textbf{Linear functions}};
\node[text width=3cm] at (6.6,0.4) {$\eta=\mathbf{a}\cdot \nabla \phi$};
\node[text width=2.6cm] at (-2.5,-0.5) {$\eta= \|\nabla \phi \|_p\\ p\geq 1, p\neq 2$};
\node[text width=3cm] at (-1.5,-1.9)  {\textbf{Norms}};
\node[text width=3cm] at (0.8,-0.5) {$\eta=   \vertiii{\nabla \phi}_{\mathbf{B}}$};
\node[text width=3cm] at (3.8,0.0) {$\eta=   \vertiii{\nabla \phi}_{\mathbf{A}}$};
\end{tikzpicture}
  \end{center}
  \caption{Overview of objective variation entropy results. The variation entropy $ \|\nabla \phi \|_p$ is not objective unless $p=2$. The matrices $\mathbf{A}$ and $\mathbf{B}$ are positive semi-definite and positive definite respectively. The vector and matrices in $\eta= \mathbf{a}\cdot \nabla \phi$, $\eta=\vertiii{\nabla \phi}_{\mathbf{A}}$ and $\eta=\vertiii{\nabla \phi}_{\mathbf{B}}$ are chosen according to (\ref{condition a}) and (\ref{condition A}), and as such we classify these variation entropies as objective.} \label{fig:tvd entropy}
\end{figure}

\subsection{Regularization of $2$-norm variation entropy}\label{subsec:Variation measured in the $2$-norm}
Due to its importance, we discuss the $2$-norm variation entropy here separately.
In particular we consider a regularized version in order to allow evaluation everywhere.
Thus we study the case where the variation entropy function is the regularized absolute value operator $\|\cdot\|_{\eps,2}: \mathbb{R}^d \rightarrow \mathbb{R}_+$ which is defined for $\mathbf{b} \in \mathbb{R}^d, \eps>0$ as:
\begin{align}\label{eq: regularized 2 norm}
  \|\mathbf{b}\|^2_{\eps,2} := \mathbf{b}\cdot \mathbf{b} +  \eps^2.
\end{align}
Notice that
\begin{align}\label{eq: regularized 2 norm 4}
 \|\mathbf{b}\|_2 \leq \|\mathbf{b}\|_{\eps,2} \leq \|\mathbf{b}\|_{2}+\eps,
\end{align}
as displayed in Figure \ref{fig: regularization norm}.
\begin{figure}[h!]
    \begin{center}
    \begin{subfigure}[b]{0.49\textwidth}
  \includegraphics[scale=0.5]{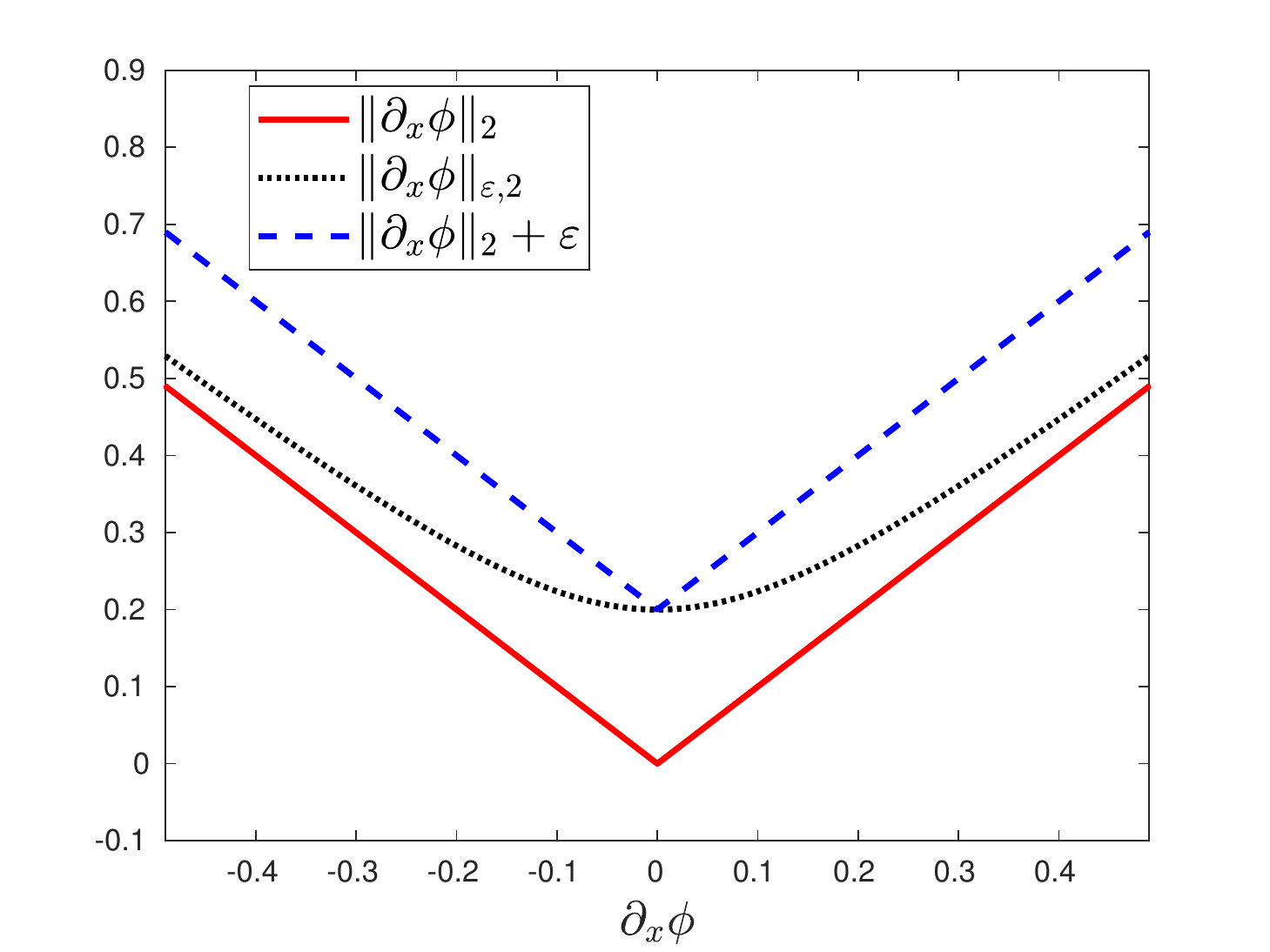}
        \caption{$1$ dimension} \label{fig:reg norm 1D}
    \end{subfigure}
    \begin{subfigure}[b]{0.49\textwidth}
     \includegraphics[scale=0.5]{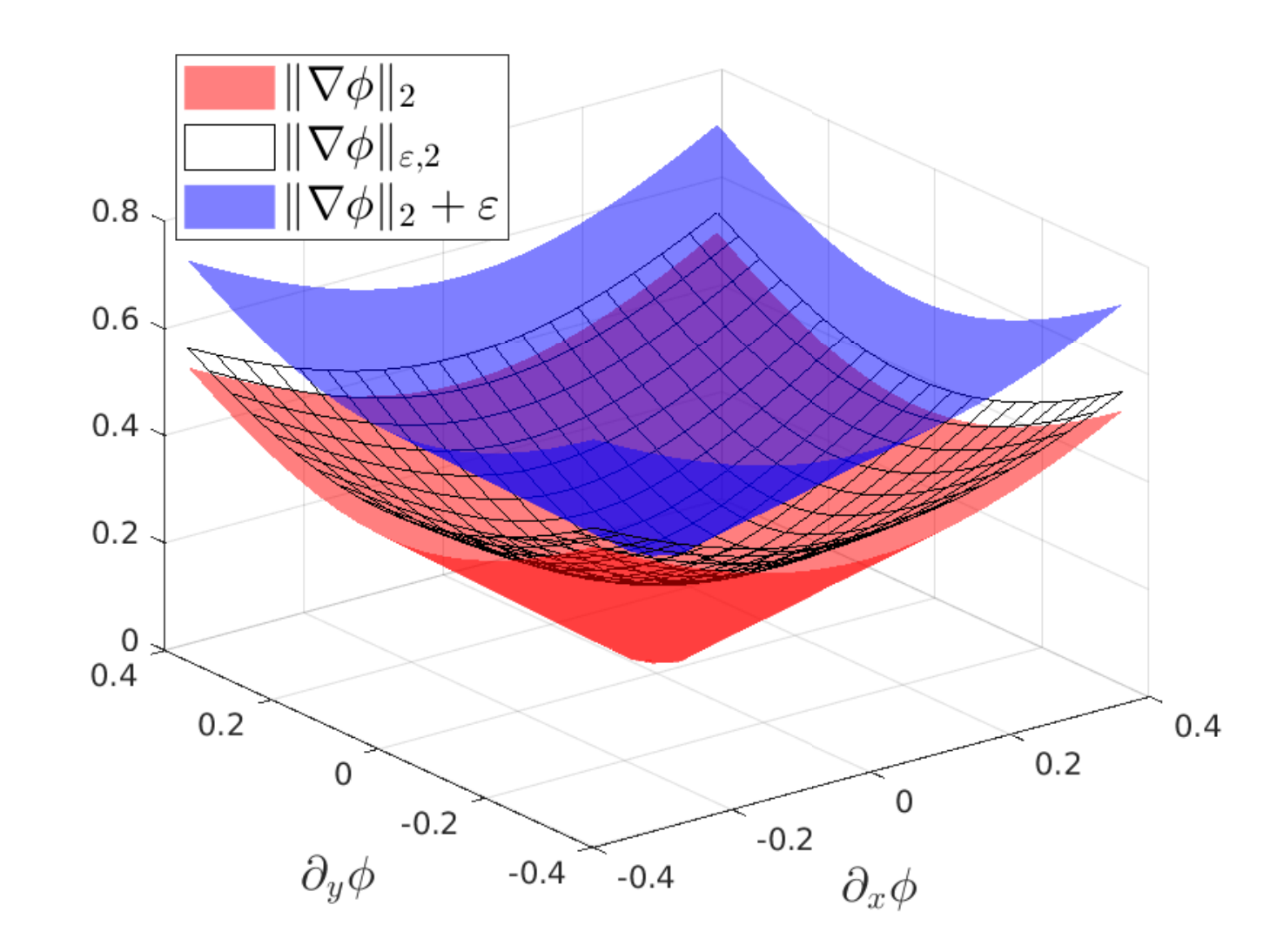}
        \caption{$2$ dimensions} \label{fig:reg norm}
    \end{subfigure}
    \caption{Plot of the regularized norm in (a) $1$-dimension and (b) $2$-dimensions. Here $\eps=0.2$.}
    \label{fig: regularization norm}
    \end{center}
\end{figure}\\
The regularized absolute value has the derivatives:
\begin{subequations}\label{eq: regularized 2 norm derivatives}
  \begin{alignat}{2}
 \partial_{\mathbf{b}} \|\mathbf{b}\|_{\eps,2} =& \frac{\mathbf{b}}{\|\mathbf{b}\|_{\eps,2}}, \label{1st derivative regularized absolute value}\\
 \partial^2_{\mathbf{b}} \|\mathbf{b}\|_{\eps,2} =& \left(\mathbf{I}-\frac{\mathbf{b}\mathbf{b}^T}{\|\mathbf{b}\|^2_{\eps,2}}\right)\frac{1}{\|\mathbf{b}\|_{\eps,2}},\label{2nd derivative regularized absolute value}
  \end{alignat}
\end{subequations}
which exist everywhere.
The homogeneity constraint (\ref{eq: homogeneity property}) is violated by a term that scales with $\eps^2$:
\begin{align}\label{eq: regularized 2 norm homogeneity}
\mathbf{b} \cdot \partial_{\mathbf{b}} \|\mathbf{b}\|_{\eps,2}  -\|\mathbf{b}\|_{\eps,2} = \frac{\eps^2}{\|\mathbf{b}\|_{\eps,2}}.
\end{align}
Also the term that appears in $\mathscr{A}$ scales with $\eps^2$:
\begin{align}\label{eq: regularized 2 norm derivatives hessian}
 \partial^2_{\mathbf{b}} \|\mathbf{b}\|_{\eps,2} \mathbf{b}=\eps^2\frac{\mathbf{b}}{\|\mathbf{b}\|^3_{\eps,2}}.
\end{align}
\begin{corollary}(Evolution equation of a regularized variation entropy)
The regularized variation $\eta=\eta_\eps = \|\nabla \phi\|_{\eps,2}$ satisfies the evolution equation:
\begin{align}\label{variation entropy evolution444}
\partial_t \eta_\eps +\nabla \cdot  \mathbf{q}_\eps =& \mathscr{D}_\eps+\mathscr{S}_\eps + \mathscr{R}_\eps ,
\end{align}
where the flux $\mathbf{q}_\eps$, the non-conservative terms $\mathscr{A}_\eps$ and $\mathscr{D}_\eps$ and the source term $\mathscr{S}_\eps$ are respectively defined as:
\begin{subequations}\label{variation entropy evolution444 terms}
  \begin{alignat}{2}
   \mathbf{q}_\eps =&  \frac{\partial \mathbf{f}}{\partial \phi}  \eta_\eps 
   +  \frac{\partial \mathbf{f}}{\partial \nabla \phi}  \nabla \eta_\eps, \\
 \mathscr{D}_\eps  =& \frac{1}{\eta_\eps^3}\left(\left(\|\nabla \phi\|_2^2\mathbf{I}-\nabla \phi \nabla \phi^T\right) \mathbf{H}_{\mathbf{x}} \phi \right): 
 \left(\frac{\partial \mathbf{f}}{\partial \nabla \phi} \mathbf{H}_{\mathbf{x}} \phi  \right),\\
 \mathscr{S}_\eps  =& \frac{\partial s}{\partial \phi} \eta_\eps,\\
 \mathscr{R}_\eps  =& \frac{\eps^2}{\eta_\eps}\left( \nabla \cdot \left(  \frac{\partial\mathbf{f}}{\partial \phi} \right) +\frac{1}{\eta_\eps^2} \mathbf{H}_{\mathbf{x}} \phi: 
 \left(\frac{\partial \mathbf{f}}{\partial \nabla \phi} \mathbf{H}_{\mathbf{x}} \phi  \right)-\frac{\partial s}{\partial \phi}\right).
  \end{alignat}
\end{subequations}
\end{corollary}
\begin{proof}
  A direct substitution of $\eta= \eta_\eps$ into (\ref{variation entropy evolution aug})-(\ref{variation entropy evolution terms}) using (\ref{eq: regularized 2 norm})-(\ref{eq: regularized 2 norm derivatives hessian}) yields
\begin{align}\label{variation entropy evolution44465}
\partial_t \eta_\eps +\nabla \cdot  \mathbf{q} =& \mathscr{A} + \mathscr{D}+\mathscr{S},
\end{align}
where the flux $\mathbf{q}$, the non-conservative terms $\mathscr{A}$ and $\mathscr{D}$ and the source term $\mathscr{S}$ are respectively given by:
\begin{subequations}\label{variation entropy evolution444654}
  \begin{alignat}{2}
   \mathbf{q} =&\frac{\partial \mathbf{f}}{\partial \phi}  \left  (\eta_\eps - \frac{\eps^2}{\eta_\eps} \right)
   +  \frac{\partial \mathbf{f}}{\partial \nabla \phi}  \nabla \eta, \\
 \mathscr{A}  =& \frac{\eps^2}{\eta^2_\eps}\frac{\nabla \phi}{\eta_\eps} \cdot
   \left( \mathbf{H}_{\mathbf{x}} \phi \frac{\partial \mathbf{f}}{\partial \phi} \right), \\
 \mathscr{D}  =& \frac{1}{\eta_\eps}\left(\left(\mathbf{I}-\frac{\nabla \phi \nabla \phi^T}{\eta_\eps^2}\right) \mathbf{H}_{\mathbf{x}} \phi \right): 
 \left(\frac{\partial \mathbf{f}}{\partial \nabla \phi} \mathbf{H}_{\mathbf{x}} \phi  \right),\\
 \mathscr{S}  =& \frac{\partial s}{\partial \phi} \left  (\eta_\eps - \frac{\eps^2}{\eta_\eps} \right).
  \end{alignat}
\end{subequations}
The divergence of the flux writes as
\begin{align}\label{eq: regularization flux term}
 \nabla \cdot \mathbf{q}  =& \nabla \cdot \left(\frac{\partial \mathbf{f}}{\partial \phi}  \left  (\eta_\eps - \frac{\eps^2}{\eta_\eps} \right)
   +  \frac{\partial \mathbf{f}}{\partial \nabla \phi}  \nabla \eta \right),\nn\\
	       =& \nabla \cdot \left(\frac{\partial \mathbf{f}}{\partial \phi}    \eta_\eps +  
   \frac{\partial \mathbf{f}}{\partial \nabla \phi}  \nabla \eta \right)+ \frac{\eps^2}{\eta_\eps^3}\nabla \phi \cdot \left(\mathbf{H}_{\mathbf{x}} \phi \frac{\partial \mathbf{f}}{\partial \phi}\right) - \frac{\eps^2}{\eta_\eps} \nabla \cdot \left(\frac{\partial \mathbf{f}}{\partial \phi}   \right)\nn\\
   =& \nabla \cdot \mathbf{q}_\eps+ \mathscr{A}- \frac{\eps^2}{\eta_\eps} \nabla \cdot \left(\frac{\partial \mathbf{f}}{\partial \phi}   \right).
\end{align}
The diffusion term $\mathscr{D}$ can be written as
\begin{align}\label{eq: regularization diffusion term}
 \mathscr{D}  =& \left(\frac{1}{\eta_\eps}\left(\mathbf{I}-\frac{\nabla \phi \nabla \phi^T}{\eta_\eps^2}\right) \mathbf{H}_{\mathbf{x}} \phi \right): 
 \left(\frac{\partial \mathbf{f}}{\partial \nabla \phi} \mathbf{H}_{\mathbf{x}} \phi  \right),\nn\\
	      =& \frac{1}{\eta_\eps^3}\left(\left(\|\nabla \phi\|_2^2\mathbf{I}-\nabla \phi \nabla \phi^T\right) \mathbf{H}_{\mathbf{x}} \phi \right): 
 \left(\frac{\partial \mathbf{f}}{\partial \nabla \phi} \mathbf{H}_{\mathbf{x}} \phi  \right) +\frac{\eps^2}{\eta_\eps^3} \mathbf{H}_{\mathbf{x}} \phi: 
 \left(\frac{\partial \mathbf{f}}{\partial \nabla \phi} \mathbf{H}_{\mathbf{x}} \phi  \right)\nn\\
 =& \mathscr{D}_\eps +\frac{\eps^2}{\eta_\eps^3} \mathbf{H}_{\mathbf{x}} \phi: 
 \left(\frac{\partial \mathbf{f}}{\partial \nabla \phi} \mathbf{H}_{\mathbf{x}} \phi  \right).
\end{align}
Substitution of (\ref{eq: regularization flux term})-(\ref{eq: regularization diffusion term}) into (\ref{variation entropy evolution44465})-(\ref{variation entropy evolution444654}) proves the claim.
\end{proof}
Consider the one dimensional case of (\ref{variation entropy evolution444})-(\ref{variation entropy evolution444 terms}), i.e.
\begin{align}\label{eq: 1D case of regularized}
\frac{\partial }{\partial t} \|\partial_x \phi\|_{\eps,2}  
+\partial_x \left ( \frac{\partial f}{\partial \phi}\|\partial_x \phi\|_{\eps,2} +  \frac{\partial f}{\partial (\partial_x \phi)} \partial_x\left( \|\partial_x \phi\|_{\eps,2}  \right)\right) 
         =&   \frac{\partial s}{\partial \phi} \|\partial_x \phi\|_{\eps,2} +\mathscr{R}_\eps .
\end{align}
We focus on the last term on the right-hand side.
It can be written as
\begin{align}
   \mathscr{R}_\eps = \left(\partial_x\left(\frac{\partial f}{\partial \phi}\right) -\frac{\partial s}{\partial \phi}\right) g_1^\eps(\partial_x \phi) + \left((\partial_{xx} \phi)^2\frac{\partial f}{\partial \partial_x \phi}\right) g_2^\eps(\partial_x \phi),
  \end{align}
  where the functions $g_1^\eps$ and $g_2^\eps$ are defined as
\begin{subequations}
  \begin{alignat}{2}
   g_1^\eps(\partial_x \phi) =&~ \frac{\eps^2}{\|\partial_x \phi\|_{\eps,2}}, \\
   g_2^\eps(\partial_x \phi) =&~ \frac{\eps^2}{\|\partial_x \phi\|_{\eps,2}^3} .
  \end{alignat}
\end{subequations}
In Figure \ref{fig: regularization} we plot $g_1^\eps$ and $g_2^\eps$ for several values of $\eps$.
\begin{figure}[h!]
    \begin{center}
    \begin{subfigure}[b]{0.49\textwidth}
  \includegraphics[scale=0.15]{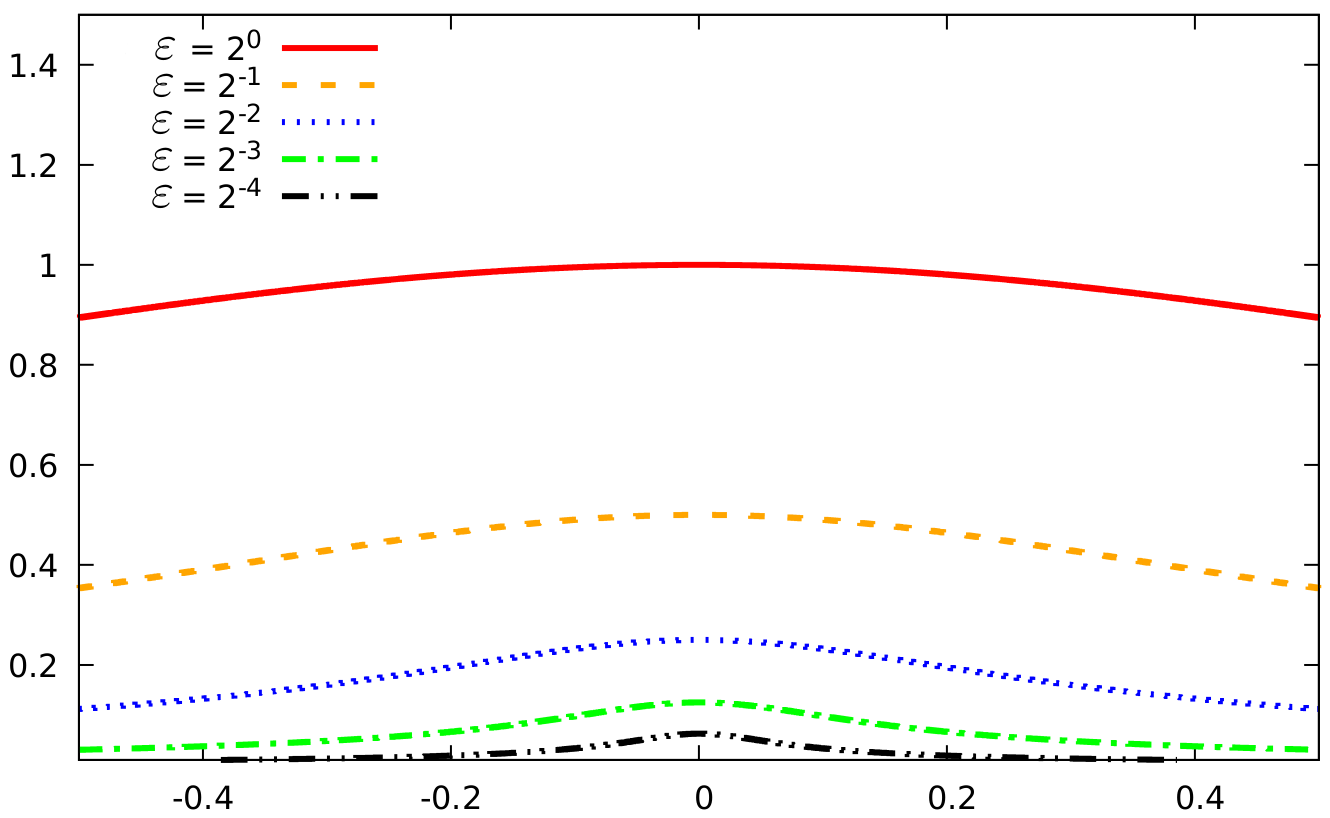}
        \caption{$g_1^\eps = g_1^\eps(\partial_x \phi)$} \label{fig:reg advection}
    \end{subfigure}
    \begin{subfigure}[b]{0.49\textwidth}
     \includegraphics[scale=0.15]{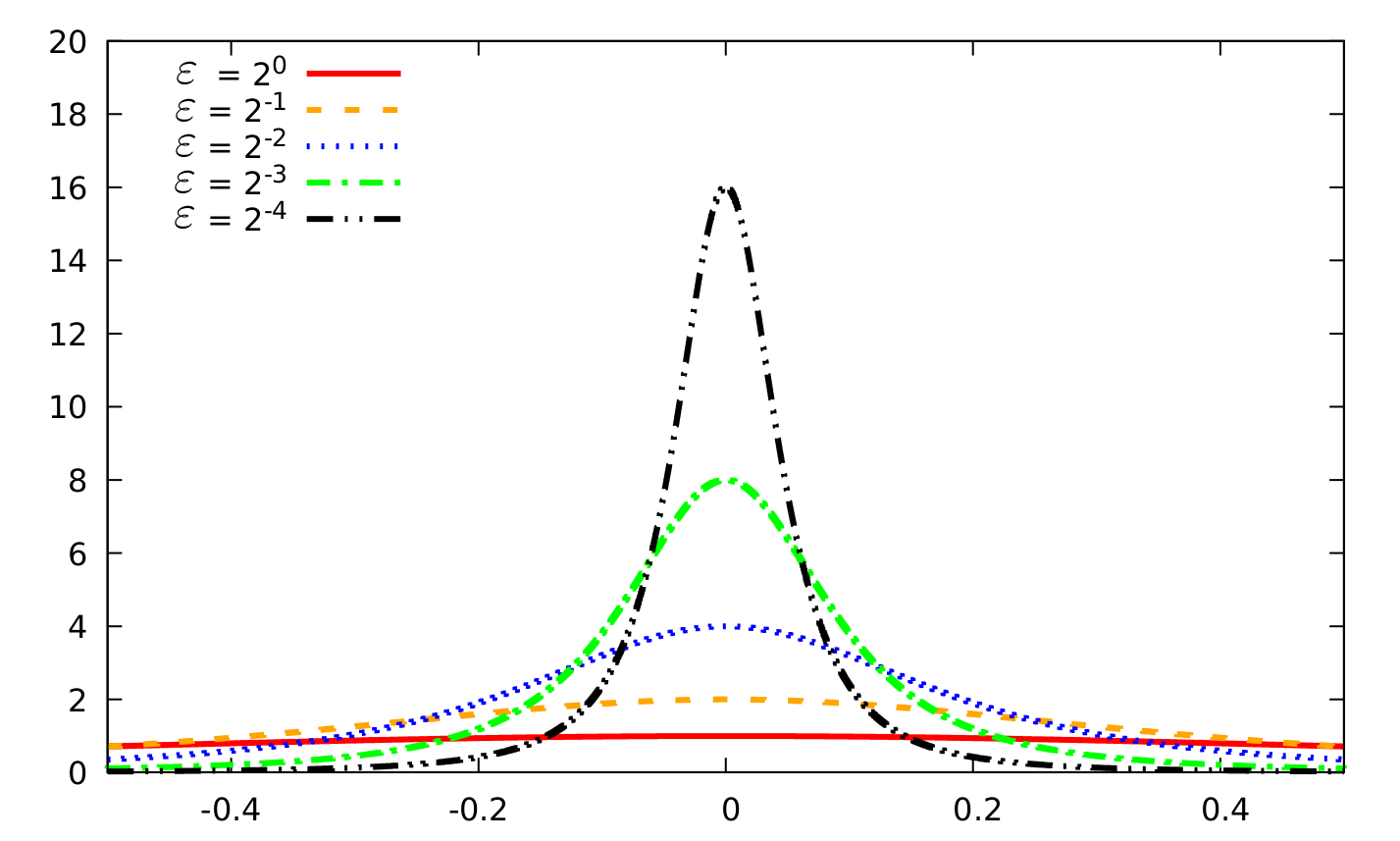}
        \caption{$g_2^\eps = g_2^\eps(\partial_x \phi)$} \label{fig:reg diffusion}
    \end{subfigure}
    \caption{Plot of the functions $g_1^\eps$ (a) and $g_2^\eps$ (b). On the horizontal axis $\partial_x \phi$ varies.}
    \label{fig: regularization}
    \end{center}
\end{figure}\\
The function $g_1^\eps$ vanishes in the limit $\eps \downarrow 0$ and its value at the origin is $g_1^\eps(0)=\eps$.
On the other hand, $g_2^\eps$ behaves as a (scaled) delta distribution centered at the origin.
The value at the origin is $g_2^\eps(0)=\eps^{-1}$ and the area under the profile $g_2^\eps$ is $2$ (which is independent of $\eps$).
Thus the regularization focuses the diffusion contribution at points where $\partial_x \phi$ approaches zero, i.e. the extrema of $\phi$.
For $\eps \downarrow 0$ we conclude that in one dimension variation entropy can either be produced by diffusion at local extrema or by the source term.

In the multi-dimensional case ($d>1$) the variation entropy diffusion $\mathscr{D}$ does not vanish. 
This is a clear separation of the $1$-dimensional case and the multi-dimensional case.
We explicitly state the temporal evolution of the non-regularized variation. 
The limit of $\eps \downarrow 0$ in (\ref{variation entropy evolution444})-(\ref{variation entropy evolution444 terms}) yields:
\begin{align}\label{variation entropy evolution4445}
\partial_t \|\nabla \phi\|_2 +\nabla \cdot  \mathbf{q} =& \mathscr{D}+\mathscr{S} ,
\end{align}
where the flux $\mathbf{q}$, the non-conservative terms $\mathscr{A}$ and $\mathscr{D}$ and the source term $\mathscr{S}$ are respectively defined as:
\begin{subequations}\label{variation entropy evolution4445 terms}
  \begin{alignat}{2}
   \mathbf{q} =&  \frac{\partial \mathbf{f}}{\partial \phi}  \|\nabla \phi\|_2 
   +  \frac{\partial \mathbf{f}}{\partial \nabla \phi}  \nabla \|\nabla \phi\|_2, \\
 \mathscr{D}  =& \frac{1}{\|\nabla \phi\|_2}\left(\left(\mathbf{I}-\frac{\nabla \phi \nabla \phi^T}{\|\nabla \phi\|_2^2}\right) \mathbf{H}_{\mathbf{x}} \phi \right): 
 \left(\frac{\partial \mathbf{f}}{\partial \nabla \phi} \mathbf{H}_{\mathbf{x}} \phi  \right),\\
 \mathscr{S}  =& \frac{\partial s}{\partial \phi} \|\nabla \phi\|_2,
  \end{alignat}
\end{subequations}
which is not defined for $\nabla \phi =\mathbf{0}$.

\begin{rmk}
 We emphasize that the well-known \textit{total variation diminishing (TVD)} constraint:
\begin{align}\label{eq: TVD constraint}
 \dfrac{{\rm d}}{{\rm d}t}\int_\Omega \| \nabla \phi(\mathbf{x},t)\|_2 ~{\rm d}\Omega \leq 0 \Rightarrow \int_\Omega \| \nabla \phi(\mathbf{x},t)\|_2 ~{\rm d}\Omega \leq \int_\Omega \|\nabla \phi_0(\mathbf{x})\|_2 ~{\rm d}\Omega,\quad \text{for}\quad t>0
\end{align}
is special case of decay of variation entropy (substitute $\eta=\|\nabla \phi\|_2$ into Corollary \ref{cor: total}).
\end{rmk}

We wish to indicate the effect of regularization on the total variation. 
To that purpose we compute the total variation and its regularized counterpart for two functions: (i) a linear approximation of the Heaviside function
\begin{align}\label{eq: lin}
  \phi_{L,\mathcal{E}}(x) = \left\{ \begin{matrix}
                                   0                    & \text{if} \quad x \leq - \mathcal{E}\\
                                   \frac{1}{2}\left(1+\frac{x}{\mathcal{E}}\right) & \text{if}  \quad x < |\mathcal{E}|\\
                                   1                    & \text{if}  \quad x \geq  \mathcal{E}
                                  \end{matrix}
  \right.,
\end{align}
and (ii) a smoothed Heaviside function:
\begin{align}\label{eq: sin}
  \phi_{S,\mathcal{E}}(x) = \left\{ \begin{matrix}
                                   0                    & \text{if} \quad x \leq - \mathcal{E}\\
                                   \frac{1}{2}\left(1+\frac{x}{\mathcal{E}} + \frac{1}{\pi}\sin\left(\frac{x \pi}{\mathcal{E}}\right)\right) & \text{if}  \quad x < |\mathcal{E}|\\
                                   1                    & \text{if}  \quad x \geq  \mathcal{E}
                                  \end{matrix}
  \right..
\end{align}
This smoothed Heaviside is often used for levelset computations, see e.g. \cite{ABKF11,AkBaBeFaKe12}.
We plot the approximated Heaviside functions and the (regularized) $2$-norms of their derivatives in Figure \ref{fig: regularization norm Heaviside}.
\begin{figure}[h!]
    \begin{center}
    \begin{subfigure}[b]{0.49\textwidth}
  \includegraphics[scale=0.5]{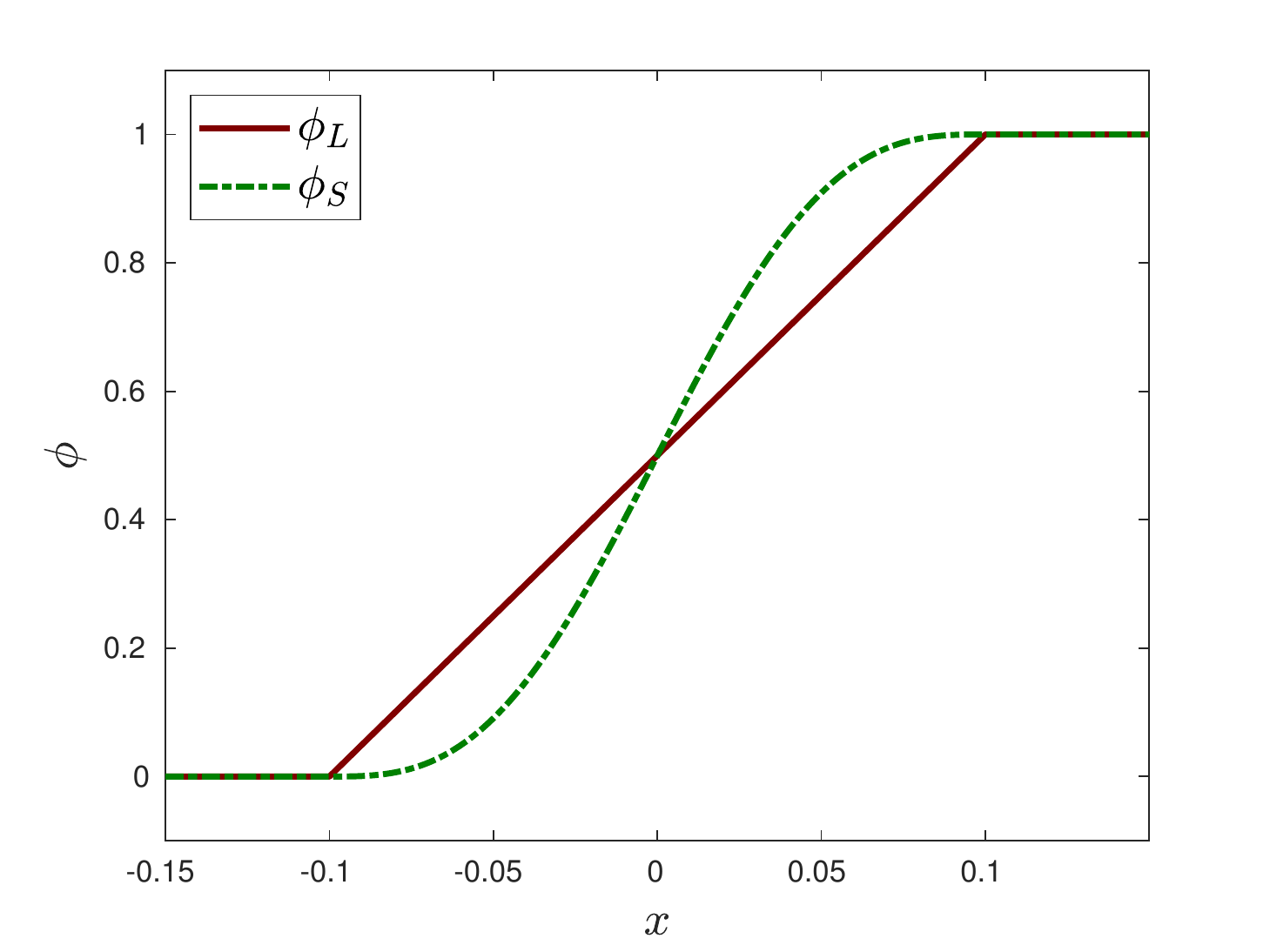}
        \caption{$\phi_{\mathcal{E}}=\phi_{\mathcal{E}}(x)$} 
    \end{subfigure}
    \begin{subfigure}[b]{0.49\textwidth}
     \includegraphics[scale=0.5]{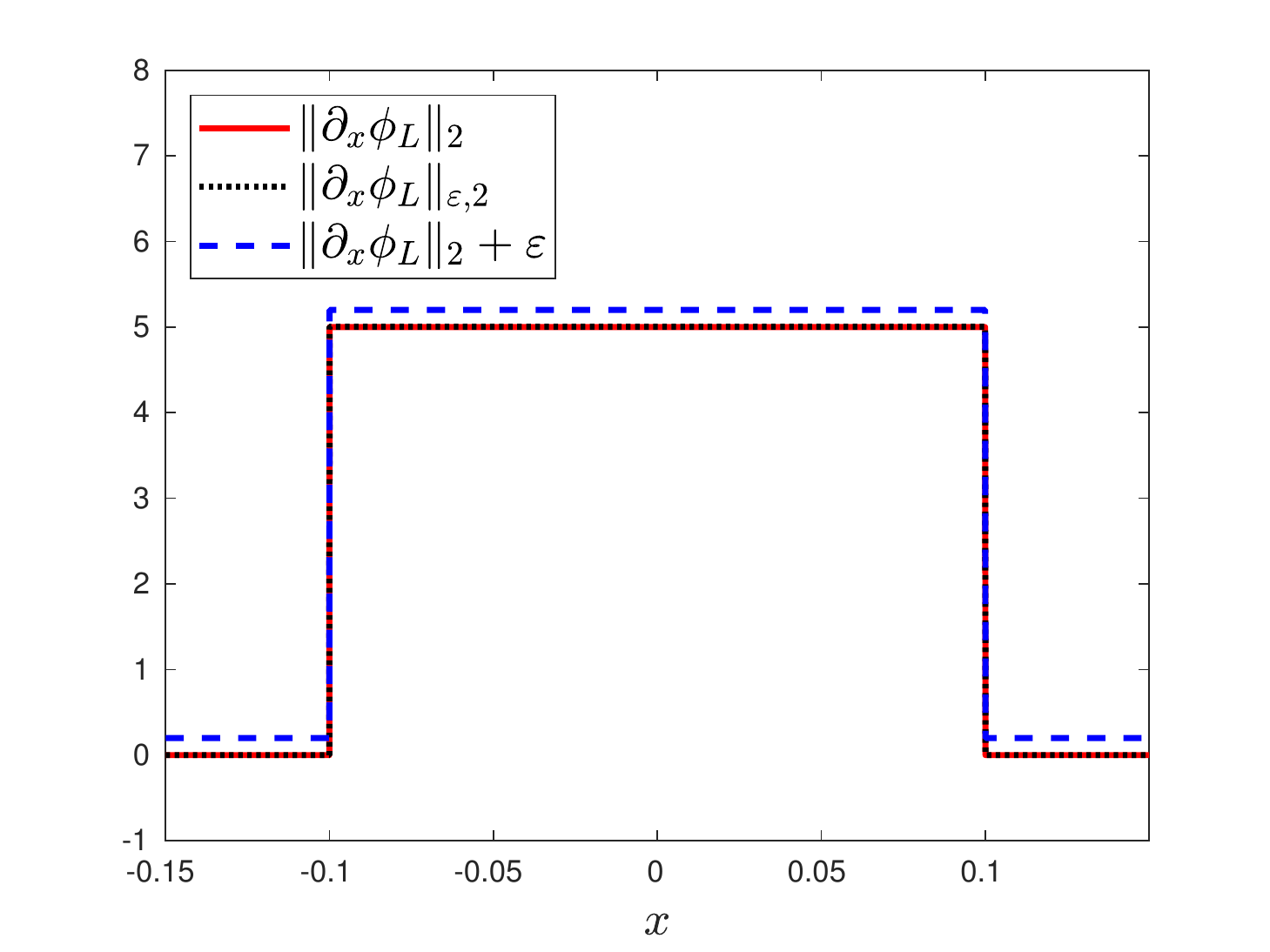}
        \caption{$\partial_x \phi_{L,\mathcal{E}}=\partial_x \phi_{L,\mathcal{E}}(x)$}
    \end{subfigure}  
    \begin{subfigure}[b]{0.49\textwidth}
     \includegraphics[scale=0.5]{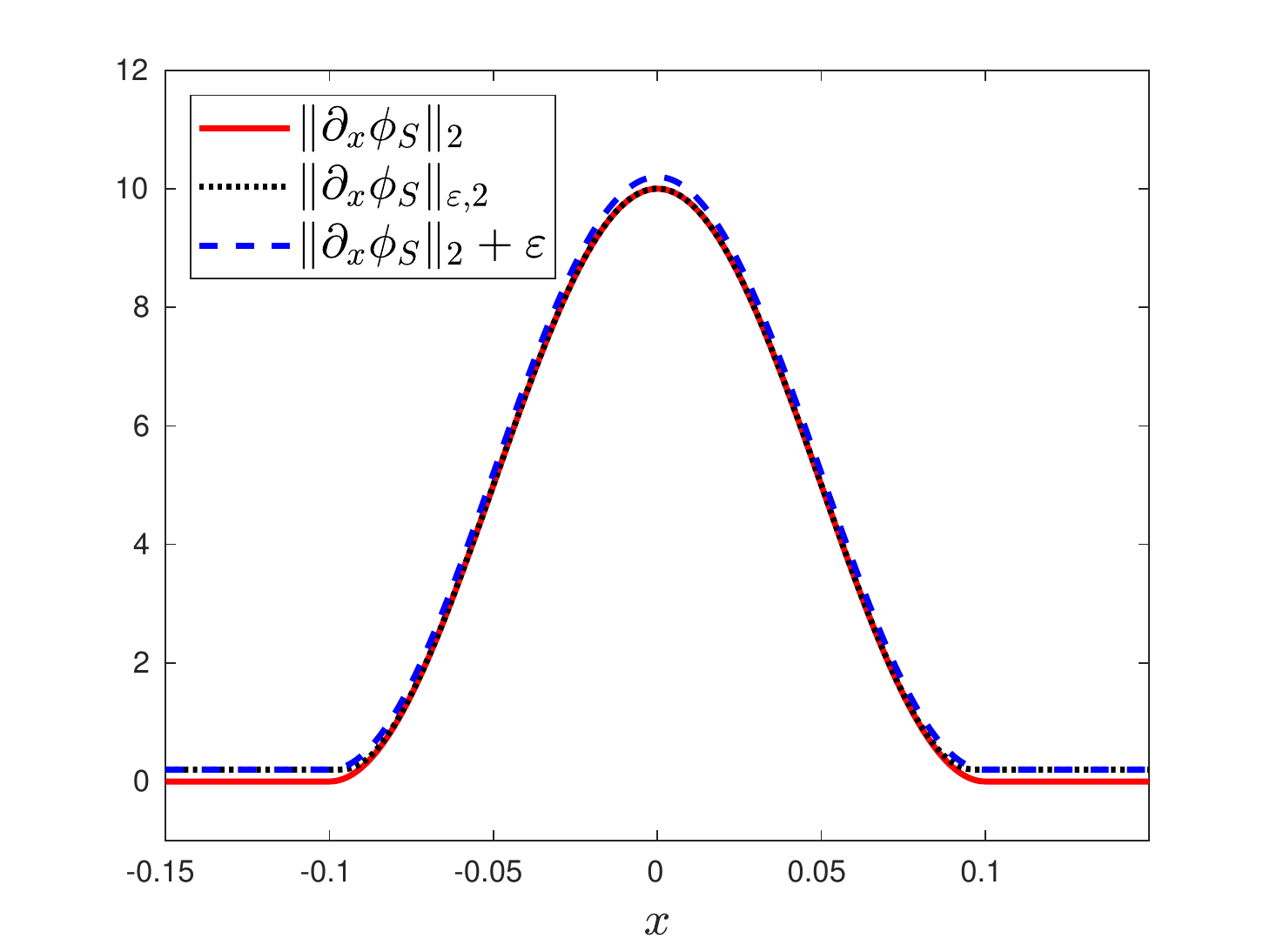}
        \caption{$\partial_x \phi_{S,\mathcal{E}}=\partial_x \phi_{S,\mathcal{E}}(x)$}
    \end{subfigure}    
    \caption{Plot of (a) the smooth Heaviside and (b) the (regularized) $2$-norms of its derivative. Here $\varepsilon = 0.2$ and $\mathcal{E}=0.1$.}
    \label{fig: regularization norm Heaviside}
    \end{center}
\end{figure}\\
The total variation and its regularizations are defined as:
\begin{subequations}\label{eq: TV Heaviside}
  \begin{alignat}{2}
TV(\phi_{\mathcal{E}}) =& \int_{-\mathcal{E}}^{\mathcal{E}} \| \partial_x \phi_{\mathcal{E}}(x)\|_2 ~{\rm d}x, \label{eq: TV normal}\\ 
TV_\varepsilon(\phi_{\mathcal{E}}) =& \int_{-\mathcal{E}}^{\mathcal{E}} \| \partial_x \phi_{\mathcal{E}}(x)\|_{\varepsilon,2} ~{\rm d}x, \label{eq: TV reg1} \\
\overline{TV}_\varepsilon(\phi_{\mathcal{E}}) =& \int_{-\mathcal{E}}^{\mathcal{E}} \| \partial_x \phi_{\mathcal{E}}(x)\|_{2} +\varepsilon ~{\rm d}x.
  \end{alignat}
\end{subequations}
The total variation (\ref{eq: TV normal}) is independent of the regularization parameter $\mathcal{E}$ since the functions (\ref{eq: lin})-(\ref{eq: sin}) monotonically increase from $0$ to $1$ and thus we have:
\begin{align}\label{eq: TV Heaviside}
 TV(\phi_{L,\mathcal{E}}) = TV(\phi_{S,\mathcal{E}})   = 1.
\end{align}
The regularized total variation (\ref{eq: TV reg1}) for the linear approximation (\ref{eq: lin}) equals:
\begin{align}\label{eq: TV Heavisidereg}
 TV_\varepsilon(\phi_{L,\mathcal{E}}) = \sqrt{1+4 \mathcal{E}^2 \varepsilon^2}.
\end{align}
Generally a small value is chosen for the parameters $\mathcal{E}$ and $\varepsilon$ which indicates that for the linear approximation the regularized total variation is nearly indistinguishable from its exact counterpart.
In Figure \ref{fig: TV regularization norm Heaviside} we show the (regularized) total variation for the both approximations.
\begin{figure}[h!]
    \begin{center}
    \begin{subfigure}[b]{0.49\textwidth}
    \includegraphics[scale=0.5]{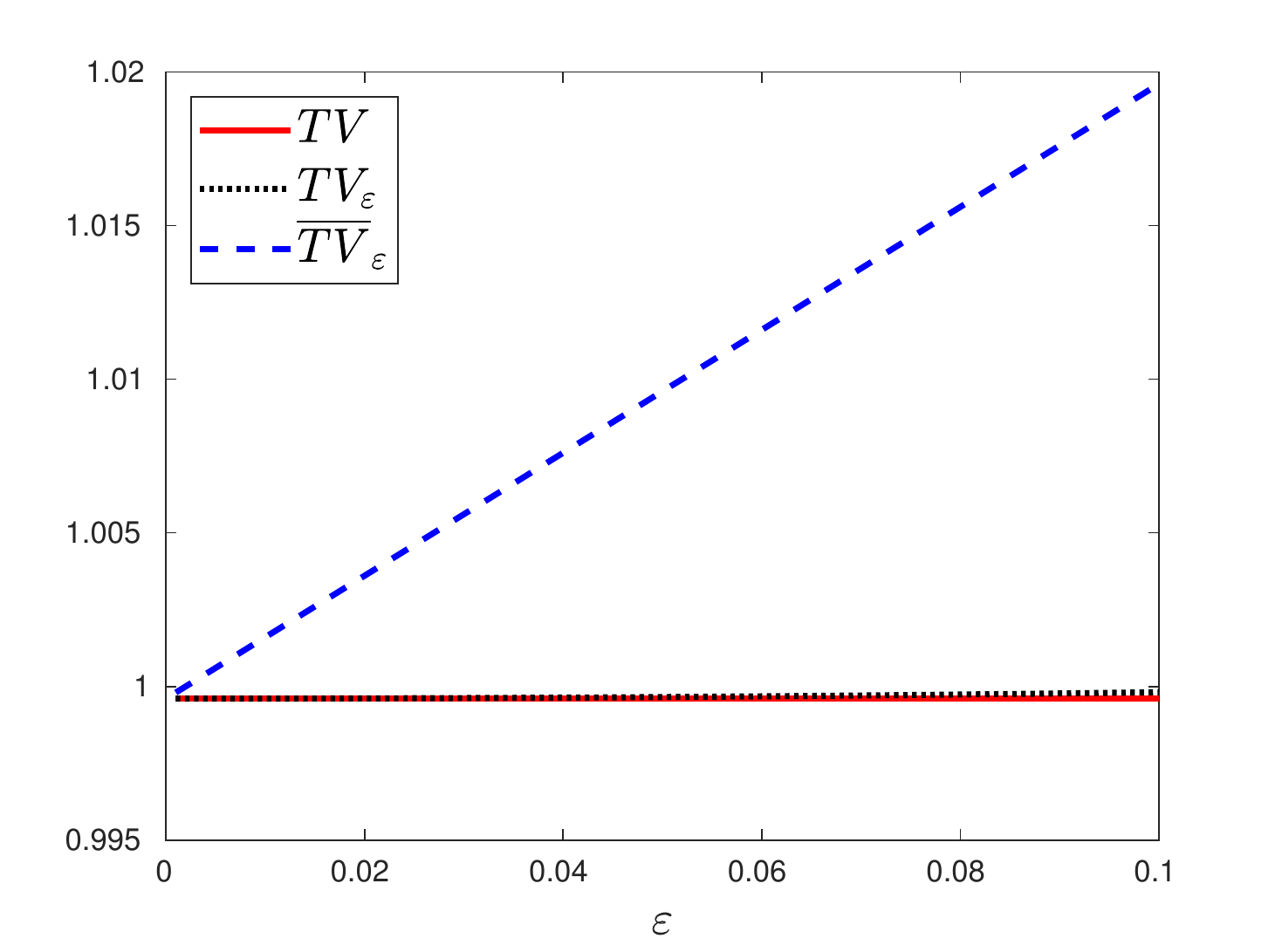}
    \caption{Total variation for the smooth Heaviside. }
    \end{subfigure}
    \begin{subfigure}[b]{0.49\textwidth}
    \includegraphics[scale=0.5]{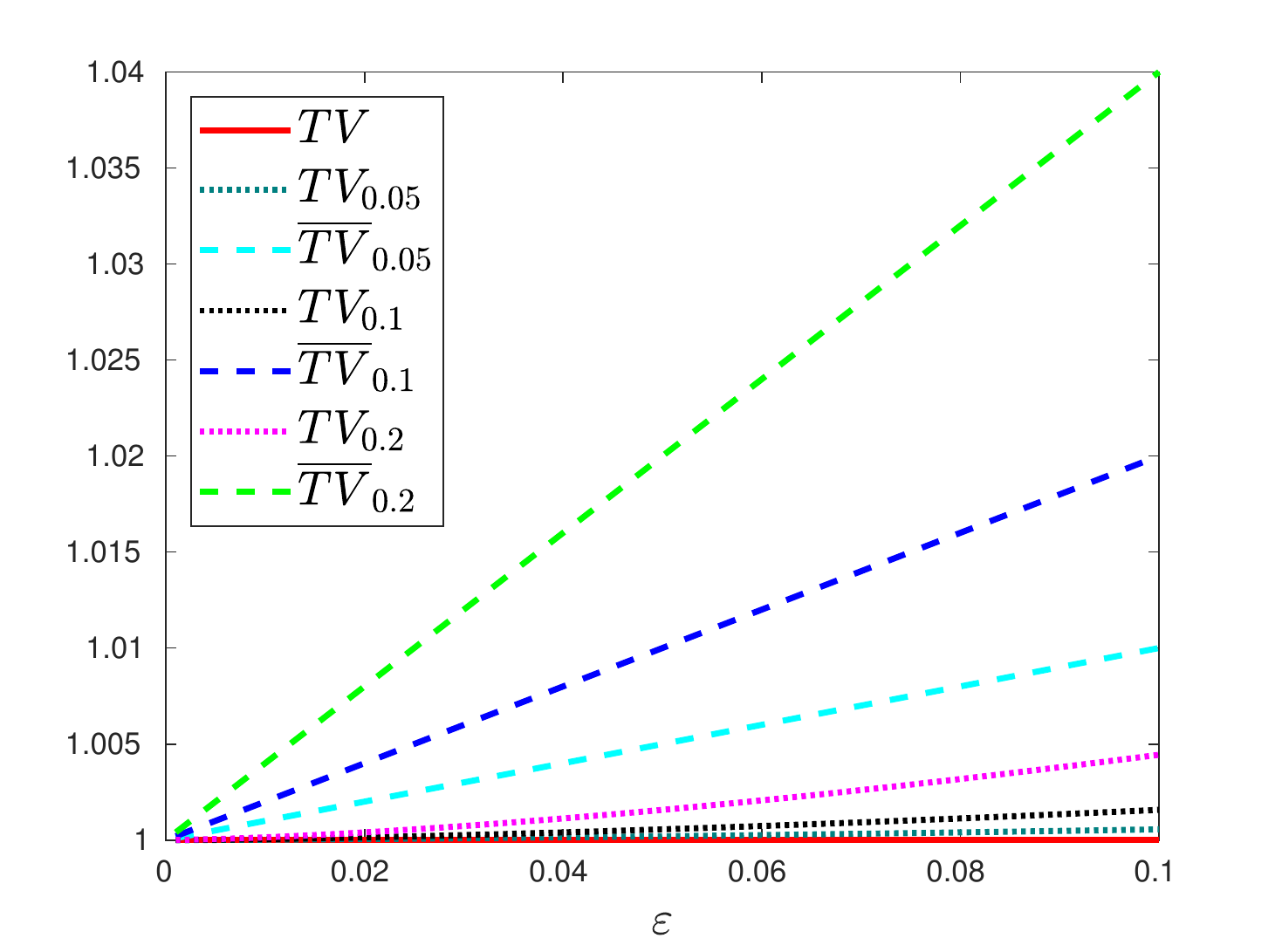}
    \caption{Total variation for the linearized Heaviside.}
    \end{subfigure}    
    \caption{Plot of total variation and its regularization for (a) the linearized Heaviside ($\mathcal{E}=0.1$) and (b) the smooth Heaviside for $\mathcal{E}=0.05, 0.1, 0.02$.}
    \label{fig: TV regularization norm Heaviside}
    \end{center}
\end{figure}\\
We see that also for the smooth Heaviside small regularization parameters $\mathcal{E}$ and $\varepsilon$ indicate a close resemblance with the exact value.
Note that the estimate
\begin{align}\label{eq: regularized TV}
 TV(\phi) \leq TV_{\varepsilon}(\phi) \leq \overline{TV}_\varepsilon(\phi),
\end{align}
which is a direct consequence of (\ref{eq: regularized 2 norm 4}), is confirmed by Figure \ref{fig: TV regularization norm Heaviside}.

\section{Conclusion and discussion}\label{sec:Conclusions}
The purpose of this paper is to answer the two questions:
\begin{itemize}
 \item \textit{How can we construct a local continuous generalization of the TVD stability condition?}
 \item \textit{Is there a connection between entropy solutions and the TVD property?}
\end{itemize}
To accomplish this we have developed the new stability concept \textit{variation entropy} for nonlinear conservation laws.
The core idea is to develop an entropy concept based on the gradient of the solution of a conservation law instead of on the solution itself.
This may be a more natural and suitable approach when dealing with shock waves, which are characterized by their large gradients.

Variation entropy solutions are formulated in the continuous and as such employing this concept eliminates the restrictions of the TVD property.
Variation entropies are homogeneous convex functions.
As a result, all semi-norms are suitable variation entropies.
A particular choice is the standard $2$-norm variation entropy which can be viewed as local continuous evolution equation of the TVD property.
This sheds light on the TVD property from the perspective of entropy solutions.
In other words, \textit{variation entropy solutions are the link between classical entropy solutions and the TVD property}.
 
This paper opens several doors for future research.
The class of variation entropy solutions can be used to design new numerical methods.
Numerical solutions in this class cannot create variation entropy and thus satisfy a local TVD stability property.
This precludes spurious oscillations which is crucial near shock waves.

A particular open problem is the construction of discontinuities capturing mechanisms in finite element methods.
A discontinuity capturing operator could be directly based on the variation entropy condition.
It could add diffusion where the variation entropy condition is harmed and render inoperative elsewhere.
We believe that a natural way to arrive at a discontinuity capturing method is via the variational multiscale (VMS) method.
In \cite{EiBaAk18TVDii} we explore this approach and use the multiscale split projector to demand the variation entropy condition.
This naturally results in a consistent VMS method equipped with a penalty term that adds diffusion when the variation entropy condition is violated.

\appendix

\section{The $3$-dimensional version of the convexity condition in spherical coordinates}\label{appendix: 3-dimensional case}
\begin{lem}\label{main thm 3d}(Convexity condition in spherical coordinates)
Let the dimension $d=3$.
The function $\hat{\eta}(r, \varphi) = F(\theta, \varphi)r$, with radius $r$ and angles $\theta, \varphi$ is convex if and only if 
the scalar-valued function $F=F(\varphi,\theta)$ satisfies
\begin{subequations}
  \begin{alignat}{2}
 A \geq B \geq 0, \label{eq: res1}
  \end{alignat}
with 
  \begin{alignat}{2}
 &A = 4\left(2 F +  \pvarphiF \cot \varphi +  \pthetathetaF \csc^2{\varphi}\right),  \\
 &B = \sqrt{2} \csc^2
     \varphi \left(\left(\pvarphiF\right)^2(1- \cos 4 \varphi)  + 32\left(\pthetaF\right)^2 + 8 \left(\pthetathetaF\right)^2 
          + 32 \left(\left(\pvarphithetaF\right)^2 - \left(\pthetaF\right)^2\right)  \sin^2 \varphi \right.  \nn\\
        &  \left. \quad\quad \quad \quad \quad    + 
         8 \left( \pvarphiF \pthetathetaF-4 \pvarphithetaF \pthetaF\right) \sin(2 \varphi)\right)^{1/2}. 
  \end{alignat}
\end{subequations}
\end{lem}
\begin{proof}
  We follow the same procedure as in the $2$-dimensional case and thus we show that the eigenvalues of the Hessian are non-negative. 
  Therefore we employ spherical coordinates:
\begin{subequations}
  \begin{alignat}{2}
 v_1 = &~r \cos \theta \sin \varphi, \\
 v_2 = &~r \sin \theta \sin \varphi, \\
 v_3 = &~r \cos \varphi. 
  \end{alignat}
\end{subequations}
The first derivatives can be written in spherical coordinates as:
\begin{subequations}
  \begin{alignat}{2}
 \dfrac{\partial}{\partial v_1} &= \cos \theta \sin \varphi \dfrac{\partial}{\partial r} - \dfrac{\sin \theta}{r \sin \varphi}\dfrac{\partial}{\partial \theta}+\dfrac{\cos \theta \cos \varphi}{r} \dfrac{\partial}{\partial \varphi},\\
 \dfrac{\partial}{\partial v_2} &= \sin \theta \sin \varphi \dfrac{\partial}{\partial r} + \dfrac{\cos \theta}{r \sin \varphi}\dfrac{\partial}{\partial \theta}+\dfrac{\sin \theta \cos \varphi}{r} \dfrac{\partial}{\partial \varphi},\\
 \dfrac{\partial}{\partial v_3} &= \cos \varphi \dfrac{\partial}{\partial r} - \dfrac{\sin \varphi}{r} \dfrac{\partial}{\partial \varphi}.
  \end{alignat}
\end{subequations}
The computation of the second derivatives is straightforward but at the same time quite involved. 
Here we only provide the resulting components of the Hessian, which are
\begin{subequations}
  \begin{alignat}{2}
 \dfrac{\partial^2}{\partial v_1^2}\eta =&  \dfrac{\sin^2 \theta}{r}F - \dfrac{\sin \theta \cos \theta}{r}\pthetaF+ \dfrac{\sin \theta \cos \theta}{r \sin^2 \varphi}\pthetaF + \dfrac{\sin^2 \theta}{r \sin^2 \varphi}\pthetathetaF+ \dfrac{\sin^2 \theta \cos \varphi}{r \sin \varphi}\pvarphiF {\color{white}affffffffggggg}\nn\\
 &- \dfrac{\sin \theta \cos \theta \cos \varphi}{r \sin \varphi}\pvarphithetaF+ \dfrac{\cos^2 \theta \cos^2 \varphi}{r}F + \dfrac{\cos\theta \cos^2 \varphi \sin\theta}{r \sin^2\varphi}\pthetaF \nn \\
 &- \dfrac{\cos\theta\sin\theta\cos\varphi}{r \sin\varphi}\pvarphithetaF + \dfrac{\cos^2\theta\cos^2\varphi}{r}\pvarphivarphiF,\\
 \dfrac{\partial^2}{\partial v_2^2}\eta =&  \dfrac{\cos^2 \theta}{r}F + \dfrac{\sin \theta \cos \theta}{r}\pthetaF - \dfrac{\sin \theta \cos \theta}{r \sin^2 \varphi}\pthetaF + \dfrac{\cos^2 \theta}{r \sin^2 \varphi}\pthetathetaF+ \dfrac{\cos^2 \theta \cos \varphi}{r \sin \varphi}\pvarphiF {\color{white}affffffffggggg}\nn \\
 &+ \dfrac{\sin \theta \cos \theta \cos \varphi}{r \sin \varphi}\pvarphithetaF+ \dfrac{\sin^2 \theta \cos^2 \varphi}{r}F - \dfrac{\cos\theta \cos^2 \varphi \sin\theta}{r \sin^2\varphi}\pthetaF \nn \\
 &+ \dfrac{\cos\theta\sin\theta\cos\varphi}{r \sin\varphi}\pvarphithetaF + \dfrac{\sin^2\theta\cos^2\varphi}{r}\pvarphivarphiF,\\
 \dfrac{\partial^2}{\partial v_3^2}\eta =& \left(F+ \pvarphivarphiF\right)\dfrac{\sin^2 \varphi}{r}, {\color{white}afefffffffffffffffffffff66666666fffffffffffggggg}\\
 \dfrac{\partial^2\eta}{\partial v_1\partial v_2} =\dfrac{\partial^2\eta}{\partial v_2\partial v_1}
					  =& + \cos^2\varphi\dfrac{\sin^2\theta-\cos^2\theta}{r \sin^2\varphi}\pthetaF -\dfrac{\sin \theta\cos\theta}{r\sin^2\varphi} \pthetathetaF - \dfrac{\sin \theta\cos\theta \cos \varphi}{r \sin \varphi}\pvarphiF  \nn\\
					   &  - \dfrac{\cos\theta \sin^2\varphi\sin\theta}{r}F  + \dfrac{(\cos^2 \theta -\sin^2\theta)\cos\varphi }{r\sin \varphi}\pvarphithetaF  + \dfrac{\sin \theta \cos \theta \cos^2\varphi}{r}\pvarphivarphiF,\\					   
  \dfrac{\partial^2}{\partial v_1\partial v_3}\eta =\dfrac{\partial^2}{\partial v_3\partial v_1}\eta=&-\dfrac{\sin \varphi \cos \theta \cos \varphi}{r}\left(F+\pvarphivarphiF\right) 
  - \dfrac{\sin \theta \cos \varphi}{r \sin \varphi}\pthetaF + \dfrac{\sin \theta}{r}\pvarphithetaF,{\color{white}afefggggg}\\
  \dfrac{\partial^2}{\partial v_2\partial v_3}\eta =  \dfrac{\partial^2}{\partial v_3\partial v_2}\eta=&-\dfrac{\sin \varphi \sin \theta \cos \varphi}{r}\left(F+\pvarphivarphiF\right) + \dfrac{\cos \theta \cos \varphi}{r \sin \varphi}\pthetaF - \dfrac{\cos \theta}{r}\pvarphithetaF.{\color{white}afefggggg}  
  \end{alignat}
\end{subequations}
The eigenvalues of the Hessian can be computed to be:
\begin{subequations}
  \begin{alignat}{2}
 \lambda_1=&0,\\
 \lambda_2=&\dfrac{1}{8 r} \left[8 F + 4 \pvarphiF \cot \varphi + 4 \pthetathetaF \csc^2{\varphi} \right. \nn \\
 & \quad \quad \left. + 
   \sqrt{2} \csc^2
     \varphi \left(\pvarphiF^2(1- \cos 4 \varphi)  + 32\pthetaF^2 + 8 \pthetathetaF^2 
          + 32 \left(\left(\pvarphithetaF\right)^2 - \left(\pthetaF\right)^2\right)  \sin^2 \varphi \right. \right. \nn\\
        & \left. \left. \quad \quad\quad \quad \quad \quad \quad    + 
         8 \left( \pvarphiF \pthetathetaF-4 \pvarphithetaF \pthetaF\right) \sin(2 \varphi)\right)^{1/2}\right], \\
 \lambda_3=&\dfrac{1}{8 r} \left[8 F + 4 \pvarphiF \cot \varphi + 4 \pthetathetaF \csc^2{\varphi} \right. \nn \\
 & \quad \quad \left. - 
   \sqrt{2} \csc^2
     \varphi \left(\pvarphiF^2(1- \cos 4 \varphi)  + 32\pthetaF^2 + 8 \pthetathetaF^2 
          + 32 \left(\left(\pvarphithetaF\right)^2 - \left(\pthetaF\right)^2\right)  \sin^2 \varphi \right. \right. \nn\\
        & \left. \left. \quad \quad\quad \quad \quad \quad \quad    + 
         8 \left( \pvarphiF \pthetathetaF-4 \pvarphithetaF \pthetaF\right) \sin(2 \varphi)\right)^{1/2}\right].       
  \end{alignat}
\end{subequations}
Positivity of the eigenvalues leads to the restrictions on $F$.
\end{proof}

\section*{Acknowledgement}\label{sec:ack}

The authors are grateful to Delft University of Technology for its support.

\bibliographystyle{unsrt}
\bibliography{refs}
\end{document}